\documentclass[reqno]{amsart}
\usepackage{graphicx} 
\usepackage{fullpage}
\usepackage{comment}
\usepackage{amssymb}

\usepackage[hidelinks]{hyperref}
\usepackage{cite}

\newtheorem{theorem}{Theorem}
\newtheorem{corollary}[theorem]{Corollary}
\newtheorem{lemma}[theorem]{Lemma}
\newtheorem{proposition}[theorem]{Proposition}

\theoremstyle{definition}
\newtheorem{notation}[theorem]{Notation}
\newtheorem{definition}[theorem]{Definition}
\newtheorem{remark}[theorem]{Remark}

\numberwithin{equation}{section}
\numberwithin{theorem}{section}

\DeclareMathOperator{\gid}{Id^{G,W}}

\DeclareMathOperator{\op}{op}
\DeclareMathOperator{\End}{End}
\DeclareMathOperator{\iden}{id}
\DeclareMathOperator{\gids}{Id^{\mathbb{Z}_2,W}}
\DeclareMathOperator{\IsupW}{Id^{\mathbb{Z}_2,W}}
\DeclareMathOperator{\VsupW}{var^{\mathbb{Z}_2,W}}
\newcommand{\gcn}{c_n^{G,W}}
\newcommand{\gcns}{c_n^{\mathbb{Z}_2,W}}
\newcommand{\mad}{\text{mod} \ }
\DeclareMathOperator{\id}{Id^{G}}
\DeclareMathOperator{\IsUT}{Id^{\mathbb{Z}_2,UT_2}}
\newcommand{\cher}{\textnormal{char}}
\DeclareMathOperator{\bi}{\mathcal{B}}
\newcommand{\W}{W \langle X \rangle}
\newcommand{\gpn}{P_n^{G,W}}
\newcommand{\gpi}{P^{G,W}_{I_1,\ldots,I_s}}
\DeclareMathOperator{\gexpG}{exp^{G,W}}
\DeclareMathOperator{\expG}{exp^G}

\newcommand{\gpr}{P^{G,W}_{n_1,\ldots,n_s}}

\newcommand{\gchi}{\chi^{\mathbb{Z}_2,W}_{r,n-r}}

\DeclareMathOperator{\spen}{span}

\usepackage{xcolor}

\title{Varieties of graded $W$-algebras and asymptotic behavior of codimension growth}

\author[G.~Busalacchi]{Giovanni Busalacchi}

\address{Giovanni~Busalacchi,  Dipartimento di Matematica e Informatica, Universit\`{a} degli Studi di Palermo, via
Archirafi 34, 90123, Palermo, Italy. {\em E-mail address}: {giovanni.busalacchi@unipa.it}}

\thanks{G. Busalacchi was supported by
Doctoral School on Mathematics and Computational Sciences, Universities of Messina, Catania and Palermo.}

\author[F.~Martino]{Fabrizio Martino}

\address{Fabrizio~Martino,  Dipartimento di Matematica e Informatica, Universit\`{a} degli Studi di Palermo, via
Archirafi 34, 90123, Palermo, Italy. {\em E-mail address}: {fabrizio.martino@unipa.it}}
\thanks{G.~Busalacchi and F.~Martino were partially supported by GNSAGA of INDAM}

\author[C.~Rizzo]{Carla Rizzo}

\address{Carla~Rizzo,  CMUC, Departamento de Matemática, Universidade de Coimbra, Largo D. Dinis, 3001-501, Coimbra, Portugal. {\em E-mail address}: {carlarizzo@mat.uc.pt}}

\thanks{C.~Rizzo was financially supported by the Fundação para a Ciência e a Tecnologia (Portuguese Foundation for Science and Technology) under the scope of the projects UID/00324/2025 (https://doi.org/10.54499/UID/00324/2025) (Centre for Mathematics of the University of Coimbra).}



\begin{document}

\begin{abstract}
    Let $W$ be a $G$-graded algebra over a field of characteristic zero, where $G$ is a finite group. We develope a theory of generalized $G$-graded polynomial identities satisfied by any finite-dimensional $W$-algebra $A$, by mean of the graded multiplier algebra of $A.$ In particular, we first prove that the graded generalized exponent exists and equals the ordinary one. Then, we explicitly compute the $G$-graded generalized identities of $UT_2,$ the $2 \times 2$ upper triangular matrix algebra equipped with its canonical $\mathbb{Z}_2$-grading, under all the possible graded $W$-actions. Finally, we exhibit examples of varieties of graded $W$-algebras with almost polynomial growth of the codimensions.
\end{abstract}

\maketitle

\section{Introduction}

Let $F$ be a field of characteristic zero and let $G$ be a finite group. Let $W$ be a unital $G$-graded algebra over $F$. An algebra $A$ over $F$ is said to be a $G$-graded $W$-algebra (or simply a \emph{graded $W$-algebra} when no ambiguity may arise) if $A$ is a graded $W$-bimodule satisfying certain additional conditions. For example, when $W = F$ this definition coincides with the standard notion of a graded $F$-algebra. In this paper, we are concerned with the polynomial identities satisfied by finite-dimensional graded $W$-algebras, with particular emphasis on the asymptotic growth of their $W$-codimension sequences.

The codimension sequence $c_n(A)$ was first introduced by Regev in \cite{Regev1972} in the context of associative algebras over fields of characteristic zero, providing a quantitative measure of the polynomial identities satisfied by a given algebra $A$. Subsequently, several properties of this sequence were established. In particular, if $A$ satisfies at least one nontrivial polynomial identity (i.e., if $A$ is a PI-algebra), then $c_n(A)$ is exponentially bounded (see \cite{Regev1972}). Moreover, as shown in \cite{Kemer1979}, its growth is either polynomial or exponential, with no intermediate behaviour. This dichotomy motivated the definition of \emph{varieties of almost polynomial growth}, namely varieties $\mathcal{V}$ whose codimension sequence grows exponentially but for which every proper subvariety $\mathcal{U} \subsetneq \mathcal{V}$ has polynomially bounded codimensions. One of the consequences of the results in \cite{Kemer1979} was the complete characterization of such varieties. Finally, the exponential growth rate of any PI-algebra was determined in \cite{GiambrunoZaicev1999}, where it was proved that the so-called PI-exponent exists and is a non-negative integer. Analogous developments were later obtained in the setting of group-graded algebras (see \cite{AG, GL, V}).

Motivated by these results, the authors of \cite{MR} carried out a detailed study of the $W$-variety generated by $UT_2$, the trivially graded algebra of $2 \times 2$ upper triangular matrices, endowed with all possible $W$-actions. Furthermore, in \cite{MR2} they established generalized versions of the aforementioned theorems for ordinary algebras. In particular, by employing the multiplier algebra--thus enabling them to refrain from specifying \emph{a priori} the acting algebra $W$ and to keep it formally defined--they classified the $W$-varieties of almost polynomial growth generated by finite-dimensional algebras and proved the existence of the PI-exponent when $W$ acts as a subalgebra of $A$ via left and right multiplication.

The aim of the present paper is to investigate the behaviour of codimension growth in the setting of $G$-graded $W$-algebras. The paper is organized as follows. Section~2 introduces the multiplier algebra $\mathcal{M}(A)$ of a graded $W$-algebra $A$, and shows that the given $G$-grading on $A$ extends naturally to $\mathcal{M}(A)$. We also establish a duality between the action of $W$ on $A$ and the corresponding action of $\mathcal{M}(A)$. Section~3 recalls the necessary background on generalized graded polynomial identities, including the definitions of the graded $W$-codimension sequence and the PI graded $W$-exponent. Section~4 proves the existence of the graded $W$-exponent and its equality with the ordinary graded PI-exponent. Sections~5 and~6 provide an extensive analysis of the graded $W$-variety generated by $UT_2$, including explicit computations of the ideals of identities, the codimension sequences, the exponents, and the graded $W$-cocharacters for all possible graded $W$-actions. Finally, Section~7 shows that the $W$-variety generated by $UT_2$ exhibits almost polynomial codimension growth whenever $W$ acts either as $F$ or as the subalgebra $D\subset UT_2$ of diagonal matrices via left and right multiplication.

\section{$W$-actions and $G$-graded algebras}
Throughout this paper, $F$ will denote a field of characteristic 0, $G=\{g_1=e_G,g_2,\ldots,g_s\}$ will be a finite group and all algebras will be associative algebras over $F$.

Let $A$ be an $F$-algebra. 
Recall that a $G$-grading on $A$ is a decomposition of $A$ into the direct sum of subspaces $A=\oplus_{g \in G}A^g$ such that $A^gA^h \subseteq A^{gh}$ for all $g,h \in G$. The subspaces $A^{g}$ are called the homogeneous components of $A$. Accordingly, an element $a \in A$ is homogeneous (or a homogeneous component of degree $g$) if $a \in A^{g}$ and in this case we write $|a|_A =g$ or simply $|a|=g$ if no ambiguity arises.  If $B$ is a subalgebra (ideal) of $A,$ then we say that $B$ is a graded subalgebra (ideal) if $B=\oplus_{g\in G} B^g$ where $B^g= B\cap A^g$ for all $g\in G.$

Given a $G$-graded algebra $W$ with unit $1_W$, we say that $A$ is a $G$-graded $W$-algebra if $A$ is a $G$-graded algebra and a $W$-bimodule, i.e., $A$ is both a left and a right $W$-module with the compatibility relation, for all $w_1,w_2\in W,$ $a\in A,$
\begin{equation}
\label{condizione 2.1}
    (w_1a)w_2=w_1(aw_2),
\end{equation}
that satisfies the following additional relations, for any $w \in W, a_1,a_2 \in A$,
\begin{equation}
\label{abw}
    (a_1a_2)w=a_1(a_2w),\ w(a_1a_2)=(wa_1)a_2,\  (a_1w)a_2=a_1(wa_2),
\end{equation}
and, for any $v \in W^{g},$ $ a \in A^h,$  $ h,g \in G$,
\begin{equation}
\label{Wgraded}
av \in A^{hg} \quad \textnormal{and} \quad  va \in A^{gh}. 
\end{equation}

 Let $\End_F(A)$ be the algebra of the endomorphisms of $A$ acting on the left of $A.$ The multiplication in $\End_F(A)$ is given by the usual composition of endomorphisms, written by juxtaposition. Moreover, we denote by $\End_F(A)^{\op}$  the opposite algebra of $\End_F(A)$,  which is the underlying vector space of $\End_F(A)$ endowed with the opposite product $\cdot^{\op}$ defined by $f_1\cdot^{\op} f_2:=f_2 f_1$ for all $f_1, f_2\in \End_F(A)$. 

 As in \cite{MR2}, we define the multipliers as follows: given $(R,L)\in\End_F(A)^{\op} \times \End_F(A),$ then $(R, L)$ is a multiplier of $A$ if for all $a,b \in A$, the following conditions hold
  \begin{align}
    R(ab) &=aR(b),  \label{condizione1} \\
    L(ab)& =L(a)b, \label{condizione2} \\
    R(a)b &=aL(b). \label{condzione3}
    \end{align}
For each $m\in A,$ the pair $(R_m,L_m)$, defined by $R_m(a)=am$ and $L_m(a)=ma$ for all $a\in A,$ is a multiplier of $A$, referred to as the inner multiplier of $A$ induced by the element $m$.  
    
We denote by $\mathcal{M}(A)$ the set of all multipliers of $A$.
This set naturally carries the structure of an 
$F$-algebra with unity, where the algebra operations are defined component-wise and the unit element is $(\iden_A, \iden_A),$
and it is called the multiplier algebra of $A$. Moreover, the set $\mathcal{IM}(A)$ of all inner multipliers of $A$ is a two-sided ideal of $\mathcal{M}(A)$ called inner multiplier ideal of $A$.

In what follows, we show that $\mathcal{M}(A)$ inherits a $G$-grading from that of $A$. To this end, we introduce the following notation.
For any $a \in A$, we write it in terms of its homogeneous components as $a = \sum_{g \in G} a^g$ with $a^g \in A^g$. For a fixed $h \in G$, we define the linear map $\pi_h: A \to A$ by
\begin{align*}
    \pi_h \left(\sum_{g \in G}a^g \right) = a^h,
\end{align*}
which is called  $h$-homogeneous linear projection, and as a consequence, we can write $a = \sum_{g \in G} \pi_g(a).$ Moreover, notice that for any $h,k\in G$ and $a,b\in A$ we have that
\begin{equation}\label{eq: prop 1 pi}
    \pi_k \pi_h(a)=\begin{cases}
        \pi_h(a) & \text{ if } k=h\\
        0 & \text{ otherwise}
    \end{cases}
\end{equation}
and, in case $a\in A^k$ or $b\in A^h,$  
\begin{equation}\label{eq: prop 2 pi}
    \pi_k(a)\pi_h(b)=\pi_{kh}(ab).
\end{equation}
With this notation in place, we introduce, for each $g \in G$ and for all $(R,L)\in\End_F(A)^{\op} \times \End_F(A),$ the following endomorphisms of $A$:
\begin{equation}\label{eq: R_g L_g}
    R_g =\sum_{h \in G}\pi_{hg}  R \pi_h \quad  \text{ and } \quad L_g =\sum_{h \in G}\pi_{gh}  L  \pi_h.
\end{equation}
\begin{lemma}
\label{multipliersgh}
  Let $g \in G$ and $(R,L) \in \mathcal{M}(A)$. Then $(R_g,L_g) \in \mathcal{M}(A)$.
\end{lemma}
\begin{proof}
     We prove only the condition \eqref{condzione3} since the same argument works also for \eqref{condizione1} and \eqref{condizione2}. To this end, without loss of generality, we may assume that $a$ and $b$ are homogeneous, so let $a=\pi_{g_1}(a) \in A^{g_1}$ and $b=\pi_{g_2}(b) \in A^{g_2}.$ Then by using \eqref{condzione3}, \eqref{eq: prop 1 pi} and \eqref{eq: prop 2 pi} we have that
$$R_{g}(a)b=\left(\sum_{h \in G}\pi_{hg} R  \pi_h\pi_{g_1}(a) \right)\pi_{g_2}(b)=\pi_{g_1g}(R(a))\pi_{g_2}(b)=\pi_{g_1gg_2}(R(a)b)=\pi_{g_1gg_2}(aL(b)).$$
    On the other hand
    $$aL_g(b)=\pi_{g_1}(a) \left(\sum_{h \in G}\pi_{gh}  L  \pi_h\pi_{g_2}(b)\right)=\pi_{g_1}(a)\pi_{gg_2}(L(b))=\pi_{g_1gg_2}(aL(b)).$$
Then $R_g(a)b=aL_g(b)$.
\end{proof}

\begin{lemma}
\label{LemmatecnicopergradsuM(A)}
For all $(R,L) \in \mathcal{M}(A)$ we have that 
$$R=\sum_{g \in G}R_g \qquad \text{and} \qquad L=\sum_{g \in G}L_g.$$
\end{lemma}
\begin{proof}
We first show that $R = \sum_{g \in G} R_g$. To this end, let $a=\sum_{h\in G}\pi_h(a)\in A$ and remark that on one side
$$
R(a) = \sum_{g\in G}\pi_g R(a) = \sum_{g\in G}\pi_gR\left(\sum_{h\in G}\pi_h(a)\right) = \sum_{g,h\in G}\pi_gR\pi_h(a).
$$
On the other, since $hG= G$ we have
$$
\sum_{g\in G}R_g(a) = \sum_{g\in G}\left(\sum_{h\in G}\pi_{hg}R\pi_h(a)\right)=\sum_{h\in G}\left(\sum_{g\in G}\pi_{hg} R\pi_h(a)\right) = \sum_{g,h\in G}\pi_gR\pi_h(a),
$$
thus $R(a) = \sum_{g\in G} R_g(a)$ as claimed. Analogously, we obtain $L = \sum_{g \in G} L_g.$
\end{proof}

\begin{proposition}
\label{gradazionediW}
    Let $A$ be a $G$-graded algebra. Then the multiplier algebra $\mathcal{M}(A)$ of $A$ is graded as follows
    $$\mathcal{M}(A)=\bigoplus_{g \in G}\mathcal{M}(A)^g,$$
    where, for each $g\in G$,
    $$
     \mathcal{M}(A)^g=\{(R,L) \in \mathcal{M}(A)\ | \ R(A^h) \subseteq A^{hg} \ \text{and} \ L(A^h) \subseteq A^{gh} \ \textnormal{for all $h \in G$}\}.
    $$
\end{proposition}
\begin{proof}
Let $(R,L)\in \mathcal{M}(A)$ and $R_g$ and $L_g$ be the endomorphism of $A$ defined in \eqref{eq: R_g L_g}. Then by Lemma \ref{LemmatecnicopergradsuM(A)} and the definition of sum in $\mathcal{M}(A),$
it follows that $(R,L)=\sum_{g \in G} (R_g,L_g).$ Moreover, by Lemma \ref{multipliersgh} and using \eqref{eq: prop 1 pi} 
we get $(R_g,L_g) \in \mathcal{M}(A)^g$ and, as a consequence, we obtain 
$\mathcal{M}(A)=\sum_{g \in G}\mathcal{M}(A)^g$.

We now prove that this sum is direct. Suppose, by contradiction, that there exists $0\neq (R,L) \in \mathcal{M}(A)^g \cap \sum_{h\neq g}\mathcal{M}(A)^h$ for some $g\in G.$ Then there must exist $h_0\neq g$ such that $R(A^k) \subseteq A^{kg} \cap A^{kh_0} = 0$ for all $k\in G,$  which implies that $R=0.$ By a similar argument, it follows that $L=0$, and thus $(R, L) =0$, contradicting our assumption. Therefore, we have proved that $\mathcal{M}(A)$ is the direct sum of the subspaces $\mathcal{M}(A)^g$ for all $g\in G$. 

Finally, it is clear that by construction
\[
\mathcal{M}(A)^g \mathcal{M}(A)^h \subseteq \mathcal{M}(A)^{gh},
\]
for all $g,h\in G,$ and the statement is proved.
\end{proof}

\begin{corollary}
    If $A$ is a $G$-graded algebra, then $\mathcal{IM}(A)$ is a graded ideal of $\mathcal{M}(A).$ In particular, $\mathcal{IM}(A)=\bigoplus_{g\in G} \mathcal{IM}(A)^g $ with $\mathcal{IM}(A)^g=\{(R_m,L_m) \, | \, m\in A^g\}$ for all $g\in G$.
\end{corollary}

As in \cite{MR2}, now we define the acting homomorphism.
The right and left actions of the algebra $W$ on $A$ naturally define the following homomorphisms of $F$-algebras:
\begin{equation*} 
    \rho \colon W \to \End_F(A)^{\op} \qquad \text{and} \qquad \lambda \colon W \to \End_F(A)
\end{equation*}
such that $\rho(w)(a) = a w$ and $\lambda(w)(a) = w a$ for all $w \in W$ and $a \in A$.

Assume now that $A$ is a $G$-graded $W$-algebra. In this setting, $A$ admits both left and right graded $W$-module structures, and one may then consider the corresponding representations $\lambda$ and $\rho$. From identity~$\eqref{abw}$, it immediately follows that, for any $w \in W$,
\[
(\rho(w), \lambda(w)) \in \mathcal{M}(A).
\]
This leads to the definition of the following homomorphism of $F$-algebras
$$
\begin{aligned}
\Phi: W &\longrightarrow \mathcal{M}(A) \\
w &\longmapsto (\rho(w), \lambda(w))
\end{aligned}
$$
which is called the acting homomorphism of $W$ on $A$. Notice that, since \eqref{Wgraded} holds, $\Phi$ is a graded homomorphism; therefore, in particular, $\Phi(W)$ is a graded subalgebra of $\mathcal M(A).$ Moreover, for all $w\in W$ by~$\eqref{condizione 2.1},$ $\Phi(W)$ satisfies the following compatibility condition
\[
\rho(w_2)\lambda(w_1) = \lambda(w_1)\rho(w_2) 
\]
for all $ w_1, w_2 \in W$. As a consequence, $A$ becomes a $G$-graded $\Phi(W)$-algebra with the action given by:
\begin{equation}
\label{Actionofphi(W)onA}
(R,L) \cdot a = L(a) \qquad a \cdot (R,L) = R(a),
\end{equation}
for any $a \in A$ and for any $(R,L) \in \Phi(W)$.

\begin{proposition}
\label{pro: unital graded W-alg}
    Let $W$ be a $G$-graded $F$-algebra. If $A$ is a unital graded $W$-algebra, then the action of $W$ on $A$ is equivalent to the action of a suitable unital graded subalgebra $B$ of $A$ by left and right multiplication.
\end{proposition}
\begin{proof}
If $A$ has unity, then by \cite[Corollary 2.5]{MR2} follows that $\mathcal{M}(A)=\mathcal{IM}(A)$ and  $A\cong \mathcal{IM}(A)$ as $F$-algebras using the isomorphism of $F$-algebras $\mu: A \to \mathcal{IM}(A)$ such that
$
\mu(m) = (R_m, L_m)
$
for any $m\in A$.
Notice that if $m \in A^g$, then for any $a\in A^h$, $R_m(a) = a m \in A^{hg}$ and $L_m(a) = m a \in A^{gh}$, so $\mu(m) \in \mathcal{IM}(A)^g$. Then $\mu(A^g)\subseteq \mathcal{IM}(A)^g$ for all $g\in G$, and $\mu$ is a $G$-graded isomorphism.
Thus, $A\cong \mathcal{IM}(A)$ as  $G$-graded algebras.

Let $\Phi$ be the acting homomorphism of $W$ on $A$. Since both $W$ and $A$ are unital, we have that $\Phi(1_W)=(R_{1_A}, L_{1_A}).$  Now, if we consider $\Phi(W) \subseteq \mathcal{M}(A)=\mathcal{IM}(A)$, then the action of $W$ on $A$ is equivalent to the action of $\Phi(W)$ on $A$ via \eqref{Actionofphi(W)onA}. Therefore, since $\Phi$ is also a graded homomorphism, it follows that the unital graded algebra $\Phi(W)$ is, through $\mu$, isomorphic to a unital graded subalgebra $B$ of $A$ that acts on $A$ by right and left multiplication. 
\end{proof}

\section{On the free $G$-graded $W$-algebra}

In this section, we first construct the free $W$-algebra. Then, we define the generalized graded codimension sequence and present some basic results that will be useful in the next section, where we will prove the first main result of the paper.

Let $ X = \{x_1, x_2, \ldots \} $ be an infinite countable set and $ W\langle X \rangle $ be the free $W$-algebra. Recall that we set 
\begin{equation}
\label{1Wx}
1_Wx1_W=x
\end{equation}
for any $x \in X$.
Also, if we fix a a vector spaces basis  $\mathcal{B}_W = \{w_i\}_{i \in \mathcal{I}}$ of $W$, then a vector space basis of $ W\langle X \rangle $ consists of the following monomials
\[
w_{i_0}x_{j_1}w_{i_1}\cdots w_{i_{n-1}}x_{j_n}w_{i_n}, \quad n \geq 1,
\]  
where $ w_{i_0}, \ldots, w_{i_n} $ are from the basis $ \mathcal{B}_W $. The product of two monomials is given first by juxtaposition and then by multiplication in $W$, and the $W$-action is given by the multiplication of $W$.

Next, we introduce a grading on $W\langle X \rangle$. We write $X=\bigcup_{g \in G} X^g$, where $X^g=\{x_1^g,x_2^g,\ldots\}$ are disjoint sets and we are requiring that the variables in $X^g$ have homogeneous degree $g,$ for all $g\in G.$ If $w_{i_0},w_{i_1},\ldots,w_{i_t}$ are homogeneous elements of $W$, the homogeneous degree of a monomial $w_{i_0}x_{j_1}^{g_{k_1}}w_{i_1}\cdots w_{i_{t-1}}x_{j_n}^{g_{k_n}}w_{i_n}$ is 
$$
|w_{i_0}|\cdot g_{k_1}\cdot |w_{i_1}|\cdots | w_{i_{n-1}}|\cdot g_{k_n}\cdot |w_{i_n}|.
$$
If we denote by $W\langle X\rangle^g$ the subspace of $W\langle X \rangle$ generated by all the monomials having homogeneous degree $g,$ then $W\langle X \rangle^gW\langle X \rangle^h \subseteq W\langle X\rangle ^{gh}$, for every $g,h$ in $G$. Therefore, it follows that
\[
W\langle X \rangle =\bigoplus_{g \in G} W \langle X \rangle ^g
\]
is a $G$-grading on $\W.$ We denote such graded $W$-algebra as $W\langle X  \rangle^G$ and we call it free $G$-graded $W$-algebra of countable rank over $F$. As free algebra, $W\langle X\rangle^G $ has the following universal property: given any $G$-graded $W$-algebra $A$, then any set theoretical map $\psi: X \to A$ such that $\psi(X^g) \subseteq A^g$ extends uniquely to a $G$-graded $W$-homomorphism $\overline{\psi}: W\langle X \rangle^G \to A$.

An element $f \in W\langle X \rangle^G$ is called $G$-graded $W$-polynomial, or graded generalized polynomial if the roles of $W$ and $G$ are clear.  We say that $f \in W \langle X \rangle^G$ is $W$-trivial if $f = 0$; otherwise, we say $f$ is $W$-nontrivial.
By naturally extending the ordinary case (when $G$ is trivial and $W=F$), the total degree of a monomial $M$ in a variable $x \in X$, is defined as the number of times the variable $x$ appears in $M$ (regardless of the exponent in $G$ and the elements of $W$).

Given a $G$-graded $W$-algebra $A$, an element $f=f(x_1^{g_1},\ldots,x_{t_1}^{g_1},\ldots,x_1^{g_s},\ldots,x_{t_k}^{g_s})$ of $W\langle X \rangle ^G$ is called $G$-graded $W$-identity of $A$, or graded generalized identity (if the roles $W$ and $G$ are clear), and we write $f \equiv 0$, if $f$ vanishes under all evaluations $x_i^{g_j} \to a_{g_j,i} \in A^{g_j}, 1 \leq j \leq s$. The set of all graded generalized identities of $A$ is an ideal denoted by
 \[
 \gid(A)=\{f \in W\langle X \rangle^G \ |\ f \equiv 0 \ \text{on} \ A \}
 \]
  which is invariant under all $G$-graded $W$-endomorphisms of $W\langle X \rangle^G$. In this case, we say that $\gid(A)$ is a $T_{G,W}$-ideal.
For any $n \geq 1$ we define
\[
 \gpn=\spen_F \{w_{i_0}x_{\sigma(1)}^{g_{j_1}}w_{i_1}\cdots w_{i_{n-1}}x_{\sigma(n)}^{g_{j_n}}w_{i_n} \ | \ \sigma \in S_n,\ g_{j_1},\ldots, g_{j_n} \in G,\ w_{i_0},\ldots,w_{i_n} \in \bi_W\}
 \]
 the space of multilinear $G$-graded generalized polynomials in the graded variables $x_1^{g_{j_1}},\ldots,x_{n}^{g_{j_n}}, g_{i_l} \in G$. Here $S_n$ stands for the symmetric group of order $n$. Since the characteristic of the base field is zero, by standard arguments, we have that the $G$-graded $W$-identities follow from the multilinear ones, and we study
 \[
 \gpn(A)=\frac{\gpn}{\gpn \cap \gid(A)}.
 \]
 The dimension
 \[
 \gcn(A)=\dim_F \gpn(A), \quad n \geq 1,
 \]
 is called the $n$th $G$-graded $W$-codimension, or graded generalized codimension, of $A$ and, roughly speaking, it gives an idea of how many  generalized polynomial identities are satisfied by $A.$

 \smallskip

  Recall that the graded codimension sequence, in the case of ordinary polynomial identities, is defined similarly. More in detail, let $F\langle X \rangle^G$ be the ordinary free $G$-graded algebra and $\id(A)=\{f \in F\langle X \rangle^G | f \equiv 0 \ \text{on} \ A\}$ be the $T_G$-ideal of ordinary $G$-graded polynomial identities of $A$. If $P_n^G$ is the space of the ordinary multilinear $G$-graded polynomials in the variables $x_1^{g_{j_1}},\ldots,x_{n}^{g_{j_n}},$ $ g_{j_l} \in G$, then $c_n^G(A)=\dim_F \frac{P_n^G}{P_n^G \cap \id(A)}$ is the $n$th ordinary graded codimension of $A$.

   If we consider $W\cong F$ then $\gpn = P_n^G$. Hence, in general, since $F \cong F1_W$ is a subalgebra of $W$, we have that $P_n^G \subseteq \gpn$. Therefore, if $A$ is a $G$-graded $W$-algebra, by following the argument of \cite[Lemma 10.1.2]{GZ} with the necessary modifications, we obtain the following relation between the ordinary graded codimensions and the generalized graded codimensions:
    \begin{equation} \label{lower}
        c_n^G(A) \leq \gcn(A), \quad n \geq 1.
    \end{equation}

 In \cite{AG,GL} the authors proved that if $A$ is a $G$-graded PI-algebra, i.e., if $A$ is a $G$-graded algebra that satisfies a non-trivial $G$-graded identity, then the limit $$\expG(A):=\lim_{n \to \infty} \sqrt[n]{c_n^G(A)}$$ 
 exists and is an integer called the $G$-graded exponent of $A$. Similarly, we define
 $$\gexpG(A):= \lim_{n \to \infty} \sqrt[n]{\gcn(A)}.$$
 One of our main goals will be to prove that this limit exists and coincides with $\expG(A)$ when $A$ is a finite-dimensional $G$-graded $W$-algebra.

 \smallskip

In the last part of this section we find a formula that simplifies the computation of the generalized graded codimension sequence.  To this end, for $n \geq 1$, consider  $ 0\leq n_1,\ldots,n_s\leq n$ such that $n_1+n_2+\cdots+n_s=n$ and
\begin{equation}
\label{LaformadelleI}
    I_1,\ldots,I_s \subseteq \{1,\ldots,n\}, \quad
    I_i \cap I_j = \emptyset \ \text{ for all } i \neq j,\quad
    |I_i|=n_i.
\end{equation}
If $I_1=\{i_1,\ldots,i_{n_1}\},\ I_2=\{i_{n_1+1},\ldots,i_{n_1+n_2}\}, \ldots,\ I_s=\{i_{n_1+\cdots+n_{s-1}+1},\ldots,i_{n_1+\cdots+n_s}\}$, we denote by $\gpi$ the vector subspace of $\gpn$ of multilinear graded generalized polynomials in the variables $x_{i_1}^{g_1},\ldots,x_{i_{n_1}}^{g_1},$ 
 $x_{i_{n_1+1}}^{g_2},\ldots, x_{i_{n_1+n_2}}^{g_2},\ldots, x_{i_{n_1+\cdots+n_{s-1}+1}}^{g_s},\ldots,x_{i_{n_1+\cdots+n_s}}^{g_s}$. Then it is clear that $\gpn$ has the decomposition
 \begin{equation}
    \label{nonso2}
\gpn = \bigoplus_{I_1,\ldots,I_s}\gpi,
\end{equation}
with $I_1,\ldots,I_s$ as in \eqref{LaformadelleI}.

Given a $G$-graded $W$-algebra $A$, if we consider $\gpr:= P^{G,W}_{\{1,\ldots,n_1\}, \ldots, \{n_1+\cdots + n_{s-1}+1, \ldots, n_1+\cdots+n_s\}}$, i.e., the vector subspace of $\gpn$ of multilinear generalized graded polynomials in the variables $x_{1}^{g_1},\ldots,x_{n_1}^{g_1},\ldots,$ $ x_{n_1+\cdots+n_{s-1}+1}^{g_s}, \ldots , x_{n_1+\cdots+n_s}^{g_s}$, then, clearly, 
    \begin{equation}
    \label{gpr}
    \frac{\gpi}{\gpi \cap \gid(A)} \cong \frac{\gpr}{\gpr \cap \gid(A)},
    \end{equation}
for all $I_1, \ldots, I_s$ as in \eqref{LaformadelleI}. 

For $n=n_1+ \cdots + n_s$, let us define
\[
c_{n_1,\ldots,n_s}^{G,W}(A):=\dim_F \frac{\gpr}{\gpr \cap \gid(A)}
\]
and denote by $\binom{n}{n_1,\ldots,n_s}=\frac{n!}{n_1! \cdots n_s!}$ the corresponding multinomial coefficient. Then we have the following result.
\begin{proposition}
\label{relg}
Let $A$ be a $W$-algebra over a field $F$ of characteristic zero. Then, for any $n \geq 1$,
$$
\gcn(A)=\sum_{\substack{n_1,\ldots,n_s\\ n_1+\cdots+n_s=n}}\binom{n}{n_1,\ldots,n_s}c_{n_1,\ldots,n_s}^{G,W}(A).
$$
\end{proposition}
\begin{proof}
Let $f\in \gpn \cap \gid(A)$. By \eqref{nonso2} we can write $f=\sum_{I_1,\ldots,I_s}f_{I_1,\ldots,I_s}$,  with $I_1,\ldots,I_s$ as in \eqref{LaformadelleI} and $f_{I_1,\ldots,I_s}\in \gpi.$ Then by standard arguments it is clear that $f_{I_1,\ldots,I_s}\in \gid(A).$
Hence
\begin{equation}
\label{nonso}
\gpn \cap \gid(A) = \bigoplus_{I_1,\ldots,I_s} \Big(\gpi \cap \gid(A)\Big),
\end{equation}
with $I_1,\ldots,I_s$ as in \eqref{LaformadelleI}.
Combining \eqref{nonso2} and \eqref{nonso} we have that
\begin{equation}
\label{nonso1}
\frac{\gpn}{\gpn \cap \gid(A)} \cong \bigoplus_{I_1,\ldots,I_s}\frac{\gpi}{\gpi \cap \gid(A)}.
\end{equation}
Since for any partition $n_1,\ldots,n_s$ of $n,$ we have $\binom{n}{n_1,\ldots,n_s}$ possibilities of $I_1,\ldots,I_s$, the claim is proved using $\eqref{gpr}$ and $\eqref{nonso1}$.
\end{proof}

\section{Generalized $G$-graded $W$-exponent}

Throughout this section, $A$ will be a finite dimensional $G$-graded $W$-algebra.
In what follows, it will also be useful the following notation: if $B\subseteq A$, then 
\[
WB  W=\left\{\sum w_i b_k  w_j \ | \  b_k \in B, \ w_i,w_j \in W\right\}.
\]
Thus, an ideal (subalgebra) $I$ of $A$ is called $W$-ideal ($W$-subalgebra) if it is $W$-invariant, i.e., $WI W\subseteq I.$ Moreover, we say that $I$ is a $G$-graded $W$-ideal ($G$-graded $W$-subalgebra) if it is a $G$-graded ideal (subalgebra)  $W$-invariant.

By the Wedderburn-Malcev Theorem for finite dimensional $G$-graded algebras (see \cite{WM}), we write
\[
A=B+J
\]
where $B$ is a maximal semisimple subalgebra of $A$, which we may assume to be $G$-graded, and $J=J(A)$ is its Jacobson radical, which is a graded ideal of $A$. Moreover, if $F$ is algebraically closed, we can write 
$$
    B=B_1\oplus  \cdots \oplus B_t
$$
where $B_1,\ldots, B_t$ are simple $G$-graded algebras, i.e., they have no proper graded ideals. Notice that in such decomposition, the semisimple part is not in general $W$-invariant. In fact, for instance, if $W\cong A$ and the action is given by left and right multiplication, since $J$ is an ideal, we have that $J B,\ B J \subseteq J$. Nevertheless, the Jacobson radical is a $W$-invariant, as proved in \cite[Theorem 3.3]{MR2} which statement, in the setting of graded algebras, is the following.
    \begin{proposition}
    Let $W$ be a $G$-graded algebra over a field $F$ of characteristic zero and let $A$ be a $G$-graded finite dimensional $W$-algebra. Then the Jacobson radical $J$ of $A$ is a $G$-graded $W$-ideal of $A$. Moreover, if $F$ is algebraically closed, then $A=B_1\oplus \cdots \oplus B_t +J,$ where $B_1,\ldots, B_t$ are $G$-graded simple algebras and $W B_iW\subseteq B_i+J$ for all $1\leq i\leq t.$
    \end{proposition}

Since the main goal of this section is to prove the existence and the integrality of the generalized graded exponent, we give the following definition that one can find in \cite[Definition 2]{GL}.

\begin{definition}
    Let $A$ be a finite dimensional $G$-graded $W$-algebra such that $A=B_1 \oplus \cdots \oplus B_t + J$ where $B_1,\ldots, B_t$ are $G$-graded simple algebras. We say that a $G$-graded subalgebra $B_{i_1} \oplus \cdots \oplus B_{i_k}$, where $B_{i_1},\ldots,B_{i_k}$ are distinct simple $G$-graded algebras, is admissible if for some permutation $(l_1,\ldots,l_k)$ of $(i_1,\ldots,i_k)$ we have that $B_{l_1}JB_{l_2}J\cdots JB_{l_k}\neq 0$.
\end{definition}
In the same paper the authors showed that $\expG(A)$ exists and equals the maximal dimension of an admissible subalgebra (see \cite[Theorem 2]{GL}). Here, we want to prove that $\gexpG(A) =\expG(A).$ To this end, let us compute an upper bound for $\gcn(A).$

Recall that $G=\{g_1,\ldots, g_s\}.$ Moreover, from now on, we let $\bi_A=\bi_{g_1} \cup \cdots \cup \bi_{g_s}$ be a basis of $A$ such that the elements of $\bi_{g_l}$ have homogeneous degree $g_l$, $1\leq l \leq s$, and $\bi_{g_1}=\{a_1,\ldots,a_{r_1}\}, \ldots, \bi_{g_s}=\{a_{r_1+\cdots +r_{s-1}+1},\ldots,a_{r_1+\cdots+r_s}\}.$ Moreover, let $n \geq 1$ and $\xi_{i,j}$, $1 \leq i \leq n, \ 1 \leq j \leq  r_1+\cdots+r_s=\dim_F A$, be commutative variables. Then we define the generic elements
\begin{equation}
    \xi_i^{g_l}=\sum_{j=r_1+\cdots+r_{l-1}+1}^{r_1+\cdots+r_l} a_j \otimes \xi_{i,j}, \ 1 \le i \le n, \ 1 \le l \le s.
    \label{eq2024}
\end{equation}
Let $\mathcal{H}$ be the algebra generated by the elements in \eqref{eq2024}. Then $\mathcal{H}$ is a $G$-graded $W$-algebra, where $\xi_i^{g_l}$ has homogeneous degree $g_l$ and, for any $w\in W$,
\begin{flalign*}
    & w  \xi_i^{g_l}=\sum_{j=r_1+\cdots+r_{l-1}+1}^{r_1+\cdots+r_l} wa_j \otimes \xi_{i,j}, \ 1 \le i \le n, \ 1 \le l \le s. \\
    &\xi_i^{g_l} w=\sum_{j=r_1+\cdots+r_{l-1}+1}^{r_1+\cdots+r_l} a_jw \otimes \xi_{i,j}, \ 1 \le i \le n, \ 1 \le l \le s.
\end{flalign*}
Notice that $\mathcal{H} \subseteq A\otimes F[\xi_{i,j}].$

\begin{proposition}\label{pro: relative graded W-algebra}
The $G$-graded $W$-algebra $\mathcal{H}$ is isomorphic to the a relatively free $G$-graded $W$-algebra of $A$ in $n\cdot s$ graded generators. 
\end{proposition}
\begin{proof}
Let 
$$\psi: \ W\langle x_1^{g_1},\ldots,x_n^{g_1},\ldots,x_1^{g_s},\ldots,x_n^{g_s} \rangle^G \to \mathcal{H}$$
be the $G$-graded $W$-homomorphism induced by mapping $x_i^{g_j} \to \xi_i^{g_j}$, $g_j \in G$, $1 \leq i \leq n$. We will prove that $\ker \psi=\gid(A)$.

Remark that if $C$ is a commutative $F$-algebra, 
then $A \otimes_F C$ is also a $G$-graded $W$-algebra where the $W$-action is defined by
    \[
     w \Big(\sum_{i \in \mathcal{I}}(b_i \otimes c_i)\Big)= \sum_{i \in \mathcal{I}} wb_i \otimes c_i \qquad \textnormal{and} \qquad \Big(\sum_{i \in \mathcal{I}}(b_i \otimes c_i) \Big) w= \sum_{i \in \mathcal{I}} b_iw \otimes c_i,
    \]
    with $b_i \in A, c_i \in C,$ for all $i\in\mathcal{I},$ and the grading is given by setting
    $
    (A \otimes_F C)^g= A^g \otimes_F C,
    $
    for all $g\in G.$
Since the characteristic of $F$ is zero, following step by step \cite[Lemma 1.4.2]{GZ}, with the necessary changes, we have that if $f$ is a graded generalized identity for $A$, then $f$ is also a graded generalized identity for $A \otimes_F C$. As a consequence, $\mathcal{H} \subseteq A\otimes F[\xi_{i,j}]$ implies that $\gid(A)=\gid(A\otimes F[\xi_{i,j}]) \subseteq \gid(\mathcal{H})$. Thus $\gid(A)  \subseteq\ker\psi$. 

Suppose now that $g=g(x_1^{g_{1}},\ldots,x_n^{g_1}, \ldots, x_1^{g_s},\ldots,x_n^{g_s}) \in \ker \psi,$ i.e., $g(\xi_1^{g_1},\ldots,\xi_n^{g_1}, \ldots, \xi_1^{g_s},\ldots,\xi_n^{g_s})=0$ in $\mathcal{H}$, and let $b_1^{g_1},\ldots,b_n^{g_1}, \ldots, b_1^{g_s},\ldots,b_n^{g_s}$ be arbitrary elements of $A$ such that $b_i^{g_l} \in A^{g_l}$, $1 \leq i \leq n, \ 1\leq l\leq s$. Write each element $b_i^{g_l}$ as a linear combination of the basis $\bi_{g_l}=\{a_{r_1+\cdots +r_{l-1}+1},\ldots,a_{r_1+\cdots+r_l}\}$ of $A^{g_l}$ and consider 
$$
b_i^{g_l}=\sum_{j=r_1+\cdots+r_{l-1}+1}^{r_1+\cdots+r_l} \mu_{i,j} a_j , \quad 1 \le i \le n, \ 1 \le l \le s,
$$
with $\mu_{i,j} \in F$. Since $F[\xi_{i,j}]$ is the free commutative algebra of countable rank, any set-theorical map $\xi_{i,r_1+\cdots+r_{l-1}+q} \to \mu_{i,r_1+\cdots+r_{l-1}+q}$  extends to a homomophism $F[\xi_{i,j}] \to F$. Hence, due to the universal property of the tensor product, the map
\[
a \to a, \ a \in A,\quad \xi_{i,r_1+\cdots+r_{l-1}+q} \to \mu_{i,r_1+\cdots+r_{l-1}+q}, \ 1 \leq i \leq n,\ 1 \leq l \leq s, \ 1 \leq q \leq r_l,
\]
extends to a homomorphism $\phi: A \otimes F[\xi_{i,q}] \to A$ such that $\phi (\xi_i^{g_l})=b_i^{g_l}, \ 1 \leq i \leq n$. Hence
\begin{align*}
    0 &= \phi (g(\xi_1^{g_1},\ldots,\xi_n^{g_1}, \ldots, \xi_1^{g_s},\ldots,\xi_n^{g_s}))=g(\phi(\xi_1^{g_1}),\ldots,\phi(\xi_n^{g_1}),\ldots,\phi(\xi_1^{g_s}),\ldots,\phi(\xi_n^{g_s}))\\
    & = g(b_1^{g_1},\ldots,b_n^{g_1}, \ldots, b_1^{g_s},\ldots,b_n^{g_s}).
\end{align*}
Since $b_1^{g_1},\ldots ,b_n^{g_s}$ are homogeneous arbitrary elements of $A$, $g(x_1^{g_1},\ldots,x_n^{g_1}, \ldots, x_1^{g_s},\ldots,x_n^{g_s}) \equiv 0$ is a graded generalized identity of $A$ and $\ker \psi\subseteq\gid(A)$ follows. 
\end{proof}

As a consequence, we have that a graded generalized polynomial $f$ is a graded generalized identity of $A$ if and only if it vanishes on the generic elements $\xi_{i}^{g}$, $g\in G$. 
Next, we will use this fact to determine an upper bound for the graded generalized codimensions.

\begin{lemma}
\label{upper}
Let $A$ be a finite dimensional $G$-graded $W$-algebra over an algebraically closed field $F$. Then there exist constants $C, u$ such that
\[
\gcn(A) \le Cn^ud^n,
\]
where $d$ is the maximal dimension of an admissible $G$-graded subalgebra.
\end{lemma}

\begin{proof}
   Let $\psi$ be the homomorphism defined in Proposition \ref{pro: relative graded W-algebra}. Since $\ker \psi= \gid(A) ,$ it is clear that,
for any $n_1,\ldots,n_s \geq 0$ such that $n_1+\cdots+ n_s=n$,
$$
c_{n_1,\ldots,n_s}^{G,W}(A)= \dim_F \mathcal{P}_{n_1,\ldots,n_s}^{G,W},
$$
where 
\begin{align*}
    \mathcal{P}_{n_1,\ldots,n_s}^{G,W}=\spen_F\{ w_{q_0}\eta_{\sigma(1)}w_{q_1} \cdots w_{q_{n-1}}\eta_{\sigma(n)}w_{q_n}\ | \ & \sigma \in S_n,\  w_{q_0},\ldots,w_{q_n} \in \bi_W, \\ & \ \eta_i=\xi_i^{g_l} \ \text{ if } n_1+\cdots + n_{l-1}+1 \leq i \leq n_1+\cdots + n_l\},
\end{align*}
i.e, $
\mathcal{P}_{n_1,\ldots,n_s}^{G,W}
$ is the vector space of multilinear graded generalized polynomials in variables
$\xi_{1}^{g_1},\ldots,\xi_{n_1}^{g_1},$ 
 $\xi_{n_1+1}^{g_2},\ldots, \xi_{n_1+n_2}^{g_2}, \ldots , \xi_{n_1+\cdots+n_{s-1}+1}^{g_s},\ldots, \xi_{n_1+\cdots+n_s}^{g_s}$.
Thus, our first aim is to compute an upper bound of $\dim_F \mathcal{P}_{n_1,\ldots,n_s}^{G,W}$. 

Taking a monomial  $ w_{q_0}\eta_{\sigma(1)}w_{q_1} \cdots w_{q_{n-1}}\eta_{\sigma(n)}w_{q_n}$ in $\mathcal{P}_{n_1,\ldots,n_s}^{G,W}$, then, by means of the definition \eqref{eq2024}, we can write it as a linear combination of elements of the type $c\otimes \xi_{1,j_1} \cdots \xi_{n,j_{n}}$, where $c$ is an element of the above basis $\mathcal{B}_A$ of $A$. Thus we shall estimate $c_{n_1,\ldots,n_s}^{G,W}(A)$ through an estimate of the number of possible monomials $ \xi_{1,j_1} \cdots \xi_{n,j_{n}}$ which can appear as coefficients of $ w_{q_0}\eta_{\sigma(1)}w_{q_1} \cdots w_{q_{n-1}}\eta_{\sigma(n)}w_{q_n}$ in the chosen basis.

Write $A=B+J$ where $B$ is a maximal semisimple $G$-graded subalgebra of $A$ and $J=J(A)$ is its Jacobson radical.  Since $F$ is algebraically closed, the semisimple part decomposes as a direct sum of simple $G$-graded algebras: $B=B_1 \oplus \cdots \oplus B_t$.  Suppose, as we may, that the above basis $\mathcal{B}_A$ of $A$  is a basis of $J$ and a basis for each of the $B_i$.
Then, by abuse of notation, since each variable $\xi_{i,j}$ in \eqref{eq2024} is attached to a basis element of some homogeneous degree, we will say that $\xi_{i,j}$ is a radical variable or a semisimple variable of some homogeneous degree.

Take a non-zero monomial $\Xi:= \xi_{1,j_1} \cdots \xi_{n,j_{n}}$ of degree $n$ that contains $i$ radical variables. 
Write $i =i_1+\cdots+i_s$ where $i_l\leq n_l$ is the number of radical variables of homogeneous degree $g_l$, for $1\leq l \leq s$. If $J^u \neq 0$ and $J^{u+1}=0$, then $\Xi$ is non-zero only if $i\leq u$ since $WJW\subseteq J$.
Thus the number of possible distributions of radical $i$ variables in a non-zero monomial is at most $\binom{n_1}{i_1}\cdots \binom{n_s}{i_s}(\dim J)^i\leq C_1 n^u$.
Now each such monomial involves $n-i$ semisimple variables with $ n_l - i_l$ variables of
homogeneous degree $g_l$, $1\leq l \leq s,$ which come from some distinct $G$-graded simple components $B_{m_1}, \ldots, B_{m_k}.$ Since $W B_k W\subseteq B_k+J$, this means that  there exists distincts $p_1, \ldots, p_q \in \{m_1,\ldots, m_k\}$ with $1\leq q \leq k$ such that the subalgebra $D=B_{p_1}\oplus \cdots \oplus B_{p_q}\subseteq B$ is admissible.

Now let $D$ be an admissible $G$-graded subalgebra of $B$ and let $d_1=| \bi_{g_1} \cap D|, d_2=|\bi_{g_2} \cap D|,\ldots, d_s=| \bi_{g_s} \cap D |$. If there are $i=i_1+\cdots+i_s$ radical variables, the number of non-zero monomials with semisimple variables coming from $D$ is at most $\binom{n_1}{i_1}\cdots \binom{n_s}{i_s}(\dim J)^i d_1^{n_1-i_i} \cdots d_s^{n_s-i_s}.$ Hence such admissible subalgebra $D$ may contribute with at most $C_2n^ud_1^{n_1}d_2^{n_2}\cdots d_s^{n_s}$ possible monomials, with $C_2$ a constant. 

Now, if $M$ is the number of admissible $G$-graded subalgebras of $B$ and $E$ is an admissible $G$-graded subalgebra of maximal dimension such that $e_1=|\mathcal{B}_1 \cap E|, \ e_2=|\mathcal{B}_2 \cap E|, \ldots, e_s=|\mathcal{B}_s \cap E|$, then we have that $e_1+\cdots+e_s=d$ and an upper bound for the number of possible non-zero monomials is $MC_2n^ue_1^{n_1}\cdots \ e_s^{n_s}$. Taking into account that we rewrote any product of $n$ basis elements of $A$ as a linear combination of basis elements, we get that
\[
c_{n_1,\ldots,n_s}^{G,W} (A) \le C_3n^ue_1^{n_1}\cdots e_s^{n_s},
\]
where $C_3$ is a constant. Thus, by Proposition \ref{relg} we have that 
\[
\gcn(A)= \sum_{\substack{n_1,\ldots,n_s \\ n_1+\cdots+n_s=n}}  \displaystyle \binom{n}{n_1,\ldots,n_s} c_{n_1,\ldots,n_s}^{G,W} (A) \le C_3n^u \sum_{\substack{n_1,\ldots,n_s\\n_1+\cdots+n_s=n}}\displaystyle \binom{n}{n_1,\ldots,n_s} e_1^{n_1}\cdots e_s^{n_s}=C_3n^ud^n,
\]
and the proof is complete.
\end{proof}

As in the ordinary case, following step by step \cite[Theorem 4.1.9]{GZ} with the necessary changes, we have that the codimensions do not change upon extension of the base field $F$ provided $F$ is infinite. So, combining \eqref{lower} and Lemma \ref{upper} with \cite[Theorem 2.3]{AG}, we have the following.
\begin{theorem}
Let $A$ be a finite dimensional $G$-graded $W$-algebra over a field $F$ of characteristic zero. Then there exists constants $C_1,C_2,r_1,r_2$ such that $C_1 \neq 0$ and 
\[
C_1n^{r_1}d^n \leq \gcn(A) \leq C_2n^{r_2}d^n,
\]
where $d=\exp^G(A)$ is the maximal dimension of an admisible $G$-graded subalgbera of $A$.
As a consequence the $G$-graded $W$-exponent $\gexpG(A)$ exists and is equal to $\exp^G(A).$ 
\end{theorem}

\section{Generalized graded identities of $UT_2(F)$}
From now on, we carry out an exhaustive study of the generalized graded polynomial identities of $UT_2:=UT_2(F)$, the algebra of $2 \times 2$ upper triangular matrices over a field $F$ of characteristic zero, together with the variety it generates. In particular, in this section we begin by classifying all possible graded $W$-actions on $UT_2$ and by computing the corresponding ideal of identities. As a consequence, we obtain the generalized graded codimension sequences and the generalized graded PI-exponents for each case under consideration.

Recall that, given a finite group $G$, the canonical grading on $UT_2$ is  defined as follows:
let $g\in G,$ $g\neq 1_G$, and set $UT_2=UT_2^1 \oplus UT_2^g$ 
where $UT_2^{1}=F e_{11} \oplus F e_{22}$ and $UT_2^{g}=Fe_{12}$.
According to \cite[Theorem 1]{VA}, up to isomorphism, $UT_2$ admits only two distinct gradings, the trivial grading and the canonical one. This readily implies that we can regard $UT_2$ as a superalgebra, i.e., graded by $\mathbb Z_2$, the cyclic group of order two.

Hence, since in what follows we will work exclusively with superalgebras, we adopt the convention that variables $y$ represent elements of $X$ of homogeneous degree $0$,  while the variables $z$ correspond to those of homogeneous degree $1$.

As is well-known, studying $UT_2$ with the trivial grading is equivalent to study the ungraded case, which has been extensively discussed in \cite{MR2, MR}. So we assume that the algebra $UT_2$ is endowed with the canonical grading.

Since the unital subalgebras of $UT_2$ are isomorphic to either $F1_{UT_2}$ or $D=Fe_{11}\oplus Fe_{22}$ or $C=F 1_{UT_2}+ Fe_{12}$ or $UT_2$ itself (see \cite[Remark 5.1]{MR2}), and each of these are also $\mathbb{Z}_2$-graded subalgebras of $UT_2,$ it follows from Proposition~\ref{pro: unital graded W-alg} that the graded algebra $\Phi(W)$ must be isomorphic to one of $F,$ $D$, $C$, or $UT_2$.
Consequently, we can define four non-equivalent structures of $\mathbb{Z}_2$-graded $W$-algebras on $UT_2$, each determined by the choice of the graded subalgebra of $UT_2$ acting by left and right multiplication.
We will denote them by $UT_2^F, UT_2^D, UT_2^C$ and simply $UT_2$ depending on whether $\Phi(W)$ is isomorphich to $F$, $D$, $C$ or the full algebra $UT_2$, respectively.

\smallskip

We start by considering $\Phi(W)\cong UT_2$ and, for simplicity of notation, we let $W=UT_2$. In other words, we are considering $UT_2$ with the action of the whole algebra $UT_2$ on itself by left and right multiplication.

 We consider as a basis the set $\bi_{UT_2}=\{1 := e_{11}+e_{22},e_{22},e_{12}\}$ and we recall that 
\[
z_1z_2 \equiv 0 \quad \text{ and } \quad [y_1,y_2] \equiv 0  
\]
are ordinary $\mathbb{Z}_2$-graded identities of $UT_2$ and by \eqref{1Wx} are also $\mathbb{Z}_2$-graded generalized identities.

Now, following the same reasoning as in \cite[Proposition 2.2]{MR}, we determine under which condition a polynomial in $1$ variable is $UT_2$-trivial.

In particular, if $\sum_{i=1}^m L_{v_i}R_{u_i}=0$ where $L_{v_i}$ and $R_{u_i}$ denote the left and right multiplication by $v_i, u_i\in W$, then $\sum_{i=1}^m v_i y u_i$ and $\sum_{i=1}^m v_i z u_i$ are $W$-trivial graded generalized polynomials. Conversely, if both  $\sum_{i=1}^m v_i y u_i$ and $\sum_{i=1}^m v_i z u_i$ are $W$-trivial for some $v_i,u_i\in W$, then $\sum_{i=1}^m L_{v_i}R_{u_i}=0$.

   According to this, the polynomials $e_{22}y-ye_{22}$, $e_{22}z$ and $ze_{22}-z$  are $UT_2$-nontrivial while $e_{12}ye_{12}$ and $e_{12}ze_{12}$ are $UT_2$-trivial. 
Moreover, a straightforward computation shows that 
$ e_{22}y-ye_{22}, \ e_{22}z, \ ze_{22}-z\in  \gids(UT_2)$.
\begin{lemma}\label{lem: cosequence}
    The generalized graded polynomials $ z_1z_2, \ e_{12}yz, \ zye_{12},\ e_{12}z$ and $ze_{12}$ are consequences of $e_{22}z$ and $ze_{22}-z$.
\end{lemma}
\begin{proof}
From $ze_{22}-z$ it follows that $z_1e_{22}z_2-z_1z_2$, then by $e_{22}z$ we obtain $z_1z_2$.
Now, by multiplying $e_{22}z $ by $e_{12}$ on the left we get $e_{12}z $. By multiplying $ze_{22}-z $ by $(-e_{12})$ on the left we get $ze_{12}$. 
    Finally, the $\mathbb{Z}_2$-graded $UT_2$-endomorphisms $\varphi_1$ and $\varphi_2$ of the free graded algebra $UT_2\langle Y,Z \rangle$ such that $\varphi_1(z)=yz$ and $\varphi_2(z)=zy$ transform the $\mathbb{Z}_2$-graded generalized polynomials $e_{12}z$ and $ze_{12}$ into $e_{12}yz$ and $zye_{12}$, respectively.
\end{proof} 

\begin{theorem}
\label{Tideale}
Let $UT_2$ be the $2\times 2$-upper triangular matrix $UT_2$-superalgebra with canonical grading, i.e., $UT_2^0= Fe_{11}\oplus Fe_{22},$ $UT_2^1= Fe_{12},$ and $UT_2$ acts on itself by left and right multiplication. 
Then
\begin{equation*}
\IsUT(UT_2)=\langle [y_1,y_2], \  e_{22}y-ye_{22}, \ ze_{22}-z, \ e_{22}z\rangle_{T_{\mathbb{Z}_2,UT_2}}
\end{equation*}
and $c_n^{\mathbb{Z}_2,UT_2}(UT_2)=2^{n-1}(n+2)+2$.
\end{theorem}
\begin{proof}
Let $T=\langle [y_1,y_2],\ e_{22}y-ye_{22}, \ e_{22}z, \ ze_{22}-z\rangle_{T_{\mathbb{Z}_2,UT_2}}$. It is clear that $T \subseteq \IsUT(UT_2),$ thus in order to reach our goal, suppose by contradiction that there exists a multilinear polynomial $f\in P_n^{\mathbb{Z}_2,UT_2}\cap\IsUT(UT_2)$ such that $f\notin T,$ i.e., $f$ is a non-zero polynomial modulo $T.$
     
     By Lemma \ref{lem: cosequence}, $z_1z_2\in T$ so we can assume that 
    \[
    f\equiv f_1(y_1,\ldots,y_n)+f_2(y_1,\ldots,y_{n-1},z) \quad (\mad T).
    \]
    Since $\cher\ F=0$, every multihomogeneous component of $f$ is still a generalized graded polynomial identity for $UT_2$, therefore $f_1$ and $f_2$ are both graded generalized identities of $UT_2$. 
    Since $[y_1,y_2]\equiv 0,$  $e_{22}y-ye_{22} \equiv 0$ and  $e_{12}ye_{12}  = 0,$ it follows that
    $$
    f_1(y_1,\ldots,y_n)\equiv \alpha y_1 \cdots y_n+ \beta e_{22} y_1\cdots y_n + \sum_{\mathcal{I},\mathcal{J}}\gamma_{\mathcal{I},\mathcal{J}}y_{i_1}\cdots y_{i_r}e_{12}y_{j_1}\cdots y_{j_{n-r}} (\mad T),
    $$
    where $\mathcal{I}=\{i_1,\ldots,i_r\}$ and $\mathcal{J}=\{j_1,\ldots,j_{n-r}\}$ are disjoint subsets of $\{1,\ldots,n\}$ such that $i_1<i_2<\ldots < i_r,$ $j_1 < j_2< \ldots < j_{n-r}$ and $0 \leq r \leq n.$

    If we substitute $y_1,\ldots,y_n$ with $1=e_{11}+e_{22}$ we get $\alpha 1 + \beta e_{22} + \sum_{\mathcal{I},\mathcal{J}}\gamma_{\mathcal{I},\mathcal{J}}e_{12}=0$, then $\alpha=\beta=0$. So 
    \[
    f_1(y_1,\ldots,y_n)\equiv \sum_{\mathcal{I},\mathcal{J}}\gamma_{\mathcal{I},\mathcal{J}}y_{i_1}\cdots y_{i_r}e_{12}y_{j_1}\cdots y_{j_{n-r}} \  (\mad T).
    \]
    For any fixed $\mathcal{I}$ and $\mathcal{J}$, if we replace $y_{i_1},\ldots,y_{i_r}$ with $e_{11}$ and $y_{j_1},\ldots,y_{j_{n-r}}$ with $e_{22}$, we get $\gamma_{\mathcal{I},\mathcal{J}}e_{12}=0$. Therefore $\gamma_{\mathcal{I},\mathcal{J}}=0$ for all $\mathcal{I}$ and $\mathcal{J},$ thus $f_1 =0$ modulo $T.$ 
    
    Now let focus our attention on $f_2.$ By Lemma \ref{lem: cosequence} $ e_{12}z, \ ze_{12},\ e_{12}yz,\ zye_{12} \in T  $.  Then from $e_{22}z \equiv 0$, $e_{22}y-ye_{22} \equiv 0$, $ e_{12}z\equiv 0$, $ e_{12}yz \equiv 0$, $ze_{22}-z \equiv 0$, $ ze_{12}\equiv 0$ and $ zye_{12}\equiv 0$ we can suppose that $e_{22}$ and $e_{12}$ do not appear in $f_2.$ Moreover $[y_1,y_2]\equiv 0,$ thus 
$$
f_2\equiv \sum_{\mathcal{I},\mathcal{J}}\alpha_{\mathcal{I},\mathcal{J}} y_{i_1}\cdots y_{i_r}zy_{j_1}\cdots y_{j_{n-1-r}} \ (\mad T),
$$
where as before $\mathcal{I}=\{i_1,\ldots,i_r\}$ and $\mathcal{J}=\{j_1,\ldots,j_{n-1-r}\}$ are disjoint subsets of $\{1,\ldots,n-1\}$ such that $i_1<i_2<\ldots < i_r,$ $j_1 < j_2< \ldots < j_{n-1-r}$ and $0 \leq r \leq n-1.$

For any fixed $\mathcal{I}$ and $\mathcal{J},$ if we substitute $y_{i_1},\ldots,y_{i_r}$ with $e_{11}$, $y_{j_1},\ldots,y_{j_{n-1-r}}$ with $e_{22}$ and $z$ with $e_{12},$ we get $\alpha_{\mathcal{I},\mathcal{J}}e_{12}= 0$. Thus $\alpha_{\mathcal{I},\mathcal{J}}=0$ for all $\mathcal{I}$ and $\mathcal{J}$ and $f_2=0$ modulo $T.$ 

This implies that $f\equiv 0 \ (\mad T),$ a contradiction. Therefore $T=\IsUT(UT_2).$

We have also proved that the following monomials form a basis of $P_n^{\mathbb{Z}_2,UT_2}$ modulo $P_n^{\mathbb{Z}_2,UT_2} \cap \IsUT(UT_2)$:
\begin{equation*}
\label{gruppodi3}
\begin{split}
        &y_1\cdots y_n,\\
       & e_{22}y_1\cdots y_n,\\
&y_{i_1}\cdots y_{i_r}e_{12}y_{j_1}\cdots y_{j_{n-r}}, \\
&y_{l_1}\cdots y_{l_u}z y_{q_1} \cdots y_{q_{n-u-1}},
\end{split}
\end{equation*}
where $0\leq r \leq n$ and $0\leq u \leq n-1$ such that $i_1<\ldots<i_r,\ j_1<\ldots<j_{n-r},\ l_1<\ldots<l_u,\ q_1<\ldots<q_{n-u-1}$.  Hence, by counting them, we get 
\begin{equation}
\label{spezzetaton}
    c_{n,0}^{\mathbb{Z}_2,UT_2}(UT_2)=1+1+ \sum_{r=0}^{n}\binom{n}{r}=2^n+2,
\end{equation}
and 
\begin{equation}
\label{spezzetaton-11}
    c_{n-1,1}^{\mathbb{Z}_2,UT_2}(UT_2)= \sum_{r=0}^{n-1}\binom{n-1}{r}=2^{n-1}.
\end{equation}
As a cosequence, by Proposition \ref{relg}, we have that
$c_n^{\mathbb{Z}_2,UT_2}(UT_2)=2^{n-1}(n+2)+2$.
\end{proof}

We now translate all the previous results in terms of the $\mathbb{Z}_2$-graded $W$-algebra $UT_2$ where $W$ acts on it as the algebra $UT_2$ by left and right multiplication. To do so, we establish the following notation.

\begin{notation}\label{notation}
    We fix an ordered basis $\bi_W=\{w_i\}_{i \in \mathcal{I}}$ of the graded algebra $W$ over the field $F$, choosing it so that the first element is $w_0=1_W$, the unity of $W.$ Whenerver we deal with a  finite-dimensional unital $G$-graded $W$-algebra $A$, where the action of $W$ is given by left and right multiplication by elements of a unital graded subalgebra $B \subseteq A$, we also fix an ordered basis $\bi_B=\{b_0=1_A,b_1,\ldots,b_n\}$ of $B$ over $F$. Under these assumptions, we assume, as we may, that the action homomorphism $\Phi$ maps each $w_i$ to the pair $(R_{b_i},L_{b_i})$ in the multiplier algebra $\mathcal{M}(A)$, for all $0\leq i\leq n$. Furthermore, for all $i \geq n+1$, we assume that $w_i$ lies in the kernel of $\Phi$, meaning $\Phi(w_i)=0$.
\end{notation}

Accordingly, fixing the ordered basis $\bi_{UT_2}=\{1 := e_{11}+e_{22},e_{22},e_{12}\}$ for $UT_2$, the following result holds.

\begin{corollary}
    Let $UT_2$ be the $\mathbb{Z}_2$-graded $W$-algebra $UT_2$ endowed with the canonical grading where $W$ acts on it as $UT_2$ itself by left and right multiplication. Then $\gids(UT_2)$ is generated, as $T_{\mathbb{Z}_2,W}$-ideal, by the following polynomials:
    \[
    [y_1,y_2], \quad   w_1y-yw_1,  \quad zw_1-z, \quad w_1z, \quad yw_i,  \quad  w_iy, \quad zw_i,  \quad w_iz 
    \]
for any $i \geq 3$. Moreover, $\gcns(UT_2)=2^{n-1}(n+2)+2$.
\end{corollary}

Now analyze the case of $\Phi(W) \cong D=Fe_{11}\oplus Fe_{22}$ and we choose the ordered basis of $D$ as the following
$$\bi_{D}=\{1_{UT_2},e_{22}\}.$$
With this assumption and similar argumentation as in Theorem \ref{Tideale}, the following theorem can be proven.
\begin{theorem}
\label{TheoremUT_2^D}
    Let $UT_2^{D}$ be the $\mathbb{Z}_2$-graded $W$-algebra $UT_2$ endowed with the canonical grading where $W$ acts on it as the algebra $D$ by left and right multiplication. Then $\gids(UT_2^D)$ is generated, as $T_{\mathbb{Z}_2,W}$-ideal, by the following polynomials:
    \[
    [y_1,y_2],  \quad w_1y-yw_1,   \quad  zw_1-z, \quad w_iz, \quad zw_j , \quad w_jy ,  \quad  yw_j
    \]
    for all $i \geq 1$, $j\geq 2$. Moreover, $\gcns(UT_2^D)=n2^{n-1}+2$.
\end{theorem}

We now study the algebra $UT_2^C,$ i.e. the $W$-algebra $UT_2$ with action given by left and right multiplication by elements of $C=F1_{UT_2}+ Fe_{12}.$ We highlight that the proof of the following result is easily obtained from the one of Theorem  \ref{Tideale}, with some mild and intuitive modifications.

To this end, let 
$$\bi_C=\{1_{UT_2}, e_{12}\}$$ a fixed ordered basis of $C.$ 
\begin{theorem}
\label{AlgebraCidentità}
    Let $UT_2^C$ be the $\mathbb{Z}_2$-graded $W$-algebra $UT_2$ endowed with the canonical grading where $W$ acts as the algebra $C$ by left and right multiplication. Then $\gids(UT_2^C)$ is generated, as $T_{\mathbb{Z}_2,W}$-ideal, by the following polynomials:
    $$[y_1,y_2], \quad z_1z_2, \quad yw_i, \quad w_iy, \quad zw_j, \quad w_jz$$
    for all $i \geq 2$ and $j \geq 1$. Moreover, $\gcns(UT_2^C)=2^{n-1}(n+2)+1$.
\end{theorem}

As a final case, for the sake of completeness, we analyze $UT_2^F.$ Recall that in this case $\Phi(W)\cong F,$ therefore we are dealing with the ordinary graded identities of $UT_2$ and we can state the following results from \cite{VA}.

\begin{theorem}
\label{TheoremUT2F}
    Let $UT_2^F$ be the $\mathbb{Z}_2$-graded $W$-algebra $UT_2$ endowed with the canonical grading where $W$ acts on it as the algebra $F$ by left and right multiplication.
    Then $\gids(UT_2^F)$ is generated, as $T_{\mathbb{Z}_2, W}$-ideal, by the following polynomials:
    \[
    [y_1,y_2], \quad z_1z_2,   \quad yw_i, \quad w_iy, \quad zw_i, \quad w_iz
    \]
    for all $i \geq 1$. Moreover, $\gcns(UT_2^F)=n2^{n-1}+1$.
\end{theorem}

\section{Graded generalized cocharacter sequence of $UT_2(F)$}

In this section, we aim to provide a description of the space of generalized multilinear polynomial identities of $UT_2,$ for each possible action of $W$, in the language of Young diagrams through the representation theory of a Young subgroup of $S_n.$

To this end, let $A$ be a $\mathbb{Z}_2$-graded $W$-algebra. For $n \geq 1$ and $0\leq r\leq n$
we consider the left permutation action of the group $S_{r} \times S_{n-r}$ on the space $P^{\mathbb{Z}_2,W}_{r,n-r}$ by letting $S_r$ act on $y_1,\ldots,y_r$ and $S_{n-r}$ on $z_1,\ldots,z_{n-r}$. Since $\gids(A)$ is invariant under this $S_r \times S_{n-r}$ action, we have that 
$$
P^{\mathbb{Z}_2,W}_{r,n-r}(A):=\frac{P^{\mathbb{Z}_2,W}_{r,n-r}}{P^{\mathbb{Z}_2,W}_{r,n-r} \cap \gids(A)}
$$
is a $S_{r}\times S_{n-r}$ left module. The corresponding character, denoted by $\gchi(A),$ is called $(r, n-r)$th generalized graded  cocharacter of $A.$

On one hand, it is well-known that the irreducible characters of the symmetric group $S_n$ are in a one-to-one correspondence with the partitions of the integer $n,$ on the other the irreducible $S_r \times S_{n-r}$ characters are obtained by taking the outer product of $S_r$ and $S_{n-r}$ irreducible characters, respectively. Then, by complete reducibility we can write
\begin{equation*}
\gchi(A)=\sum_{\substack{\lambda \vdash r \\ \mu \vdash n-r}}m_{\lambda,\mu}(\chi_\lambda \otimes \chi_\mu)
\end{equation*}
where $\chi_\lambda$ (respectively, $\chi_{\mu}$) denotes the irreducible $S_r$-character (respectively $S_{n-r}$-character) and $m_{\lambda,\mu} \geq 0$ are the corresponding multiplicities. 

Recall that the multiplicities in the cocharacter sequence are equal to the maximal number of linearly independent highest weight vectors, according to the representation theory of $GL_n.$ We also recall that a highest weight vector is obtained from the polynomial corresponding to an essential idempotent by identifying the variables whose indices lie in the same row and alternating the ones whose indices lie in the same column of the corresponding Young tableau (see \cite{Drenskybook} for more details).
Furthermore
\begin{equation}
\label{gcrn-r}
    c_{r,n-r}^{\mathbb{Z}_2, W}(A)= \sum_{\substack{\lambda \vdash r \\ \mu \vdash n-r}}m_{\lambda,\mu}d_{\lambda}d_{\mu}
\end{equation}
where $d_{\lambda}$ and $d_\mu$ are the degrees of $\chi_\lambda$ and $\chi_\mu$, i.e., the number of standard tableaux of shape $\lambda$ and $\mu,$ respectively.

In what follows, given a partition $\lambda\vdash n,$ we define height of $\lambda,$ and we write $h(\lambda),$ as the number of rows of $\lambda.$ We can immediately remark that if $\dim A^0 = k$ and $\dim A^1= t,$ then every generalized polynomial alternating in $k$ or more even variables or $t$ or more odd variables is a polynomial identity, thus $m_{\lambda, \mu}=0$ whenever $h(\lambda)>k$ or $h(\mu)>t.$

\smallskip

Next, we compute the decomposition into irreducible components of the generalized graded cocharacter of the superalgebra $UT_2$ endowed with the canonical grading.

We start, as in the previous section, considering $\Phi(W)\cong UT_2$ and for the simplicity of the notation $W=UT_2$. Then we prove the following lemmas for the multplicities $m_{\lambda,\mu}$ of the $(r,n-r)$th generalized graded cocharacter of $UT_2$
\begin{equation}
\label{chigenerale}
\chi_{r,n-r}^{\mathbb{Z}_2,UT_2}(UT_2)=\sum_{\substack{\lambda \vdash r \\ \mu \vdash n-r}}m_{\lambda,\mu}(\chi_\lambda \otimes \chi_\mu).
\end{equation}
\begin{lemma}
\label{lemmam_1}
    If $r=n$, then in \eqref{chigenerale} we have
    $$
    m_{\lambda,\mu}=\begin{cases}
   n+3 &\text{if }\lambda=(n), \ \mu=\emptyset\\
   q+1 &\text{if }\lambda=(p+q,p), \ \mu=\emptyset \\
     0 & \text{otherwise}
\end{cases}
    $$
    for all $p \geq 1$ and $q \geq 0$ such that $2p+q=n$.
\end{lemma}
\begin{proof}
Since $\dim UT_2^0=2,$ then $m_{\lambda,\mu}=0$ whenever $h(\lambda)  >  2.$
If $\lambda=(n)$ and $\mu=\emptyset$, we consider the following tableau
\[
\begin{tabular}{|c|c|c|c|}
\hline
1 & 2 & $\cdots$ & $n$ \\
\hline
\end{tabular}
\]
and we associate to it the polynomials
\begin{equation*}
\label{eq1}
    y^n,
\end{equation*}
\begin{equation*}
\label{eq2}
e_{22}y^n,
\end{equation*}
\begin{equation*}
    \label{eq3}
    y^ie_{12}y^{n-i},
\end{equation*}
for all $0\leq i\leq n.$
It is clear that such polynomials are not generalized identities of $UT_2$ and moreover they are linearly independent modulo $\IsUT(UT_2).$ Indeed, if
\[
f(y)=\alpha y^n + \beta e_{22}y^n + \sum_{i=0}^n  \gamma_i y^i e_{12} y^{n-i}
\in\ \IsUT(UT_2)
\]
and we evaluate $y$ in 1, we have \[
\alpha + \beta e_{22} + \sum_{i=0}^n \gamma_i e_{12}=0
\]
so $\alpha=\beta=0$. Therefore, if we consider $0 \neq \delta \in F$ and $y=\delta e_{11}+ e_{22}$ then
\[
f(\delta e_{11} + e_{22})= \sum_{i=0}^n \gamma_i(\delta e_{11} + e_{22})^i e_{12} (\delta e_{11} + e_{22})^{n-i}= \sum_{i=0}^n \gamma_i \delta^i e_{12}
=0\]
Since the matrix associated to the above linear system is a Vandermonde matrix and the field $F$ is infinite, it follows that $\gamma_{i}=0$ for all $0\leq i\leq n$. This implies that $m_{\lambda,\mu} \geq n+3.$

Now let $\mu=\emptyset$ and  $\lambda=(p+q,p)$ with $p \geq 1, q \geq 0$ and $2p+q=n$. For all $0\leq i\leq q,$ we consider the following tableau:
\[
T_\lambda^i=\begin{tabular}{|c|c|c|c|c|c|c|c|c|c|}
\hline
$i+1$ & $i+2$ & $\cdots$ & $i+p$ & 1 & $\cdots$ & $i$ & $i+2p+1$ & $\cdots$ & $n$ \\
\hline
$i+p+1$ & $i+p+2$ & $\cdots$ & $i+2p$ & \multicolumn{5}{c}{} \\
\cline{1-4}
\end{tabular}
\]
and its associate polynomials
\begin{equation}
\label{eq4}
c^{(i)}_{p,q}(y_1,y_2)=y_1^i\underbrace{\overline{y}_1\cdots\tilde{y}_1}_{p-1}(y_1e_{12}y_2-y_2e_{12}y_1)\underbrace{\overline{y}_2\cdots \tilde{y}_2}_{p-1}y_1^{q-i}
\end{equation}
where the symbol $\tilde{\phantom{a}}$ or $\overline{\phantom{a}}$ means alternation on the corresponding variables. Moreover, the polynomials in \eqref{eq4} are linearly independent modulo $\IsUT(UT_2)$. Indeed if 
\[
\sum_{i=0}^q\alpha_ic_{p,q}^{(i)}(y_1,y_2) \in \IsUT(UT_2),
\]
then by substituting $y_1=\delta e_{11} + e_{22}$, where $\delta \in F,$ $\delta \neq 0$, and $y_2=e_{11}$, we get 

$$
\sum_{i=0}^q(-1)^{p}\alpha_i\delta^i=0. 
$$

Again, by a Vandermonde argument, we have that $\alpha_i=0$ for any $0\leq i\leq q.$ It follows that the polynomials in \eqref{eq4} are linearly independent (mod $\gids(A)$) and this implies that $m_{\lambda, \mu}\geq q+1.$

By taking into account \eqref{spezzetaton} and \eqref{gcrn-r}, we have proved that
$$
2^n+2=c_{n,0}^{\mathbb{Z}_2, UT_2}(UT_2)\geq (n+3)d_{(n)}d_\emptyset + \sum_{\substack{ 1 \leq p \leq \lfloor\frac{n}{2}\rfloor \\ 0 \le q \le n-2p}} (q+1) d_{(p+q,p)}d_\emptyset, 
$$
so far.

Now, in the proof of \cite[Theorem 4.4]{MR}, it was showed that
\begin{equation}
\label{Rizzo}
\sum_{\substack{1 \leq p \leq \lfloor{\frac{n}{2}}\rfloor \\ 0 \leq q \leq n-2p}} (q+1) d_{(p+q,p)}=2^n-n-1,
\end{equation}
therefore, recalling that $d_\emptyset=1,$  we get
\[
 (n+3)d_{(n)}d_\emptyset + \sum_{\substack{ 1 \leq p \leq \lfloor\frac{n}{2}\rfloor \\ 0 \le q \le n-2p}} (q+1) d_{(p+q,p)}d_\emptyset=n+3 +2^n-n-1=2^n+2.
\]
This readily implies that $m_{(n), \emptyset} = n+3$ and $m_{(p+q,p), \emptyset}= q+1,$ as required.
\end{proof}
\begin{lemma}
\label{lemmam_2}
    If $r=n-1$, then in $\eqref{chigenerale}$ we have
    $$
        \label{m_2}
         m_{\lambda,\mu}=\begin{cases}
   n &\lambda=(n-1), \ \mu=(1)\\
 q+1 &\lambda=(p+q,p), \ \mu=(1) \\
     0 & \text{otherwise}
\end{cases}
    $$
    for all $p \geq 1,q \geq 0$ such that $2p+q=n-1$
\end{lemma}
\begin{proof}
    As in Lemma \ref{lemmam_1}, we have $m_{\lambda,\mu}=0$ as long as $h(\lambda)  >  2.$ Now, let $\lambda=(n-1).$ Then for any $0\leq i\leq n-1$ we consider the following two tableaux
\[
T_\lambda^{(i)}=\begin{tabular}{|c|c|c|c|c|c|}
\hline
$1$ & $\cdots$ & $i$ & $i+2$ & $\cdots$ & $n$ \\
\hline
\end{tabular}
\]
\[
T_\mu^{(i)}=\begin{tabular}{|c|}
\hline
$i+1$ \\
\hline
\end{tabular}
\]
and the polynomial associated to them
$$
    a^{(i)}(y,z)=y^{i}zy^{n-i-1}.
$$
The above polynomials are not generalized identities, moreover they are linearly independent modulo the generalized graded identities of $UT_2,$ in fact if we suppose the contrary, i.e. $\sum_{i=0}^{n-1}\alpha_ia^{(i)}\equiv 0 \pmod{\IsUT(UT_2)}$ with some non-zero scalar $\alpha_i,$ then we let $t$ be the greatest integer such that $\alpha_t\neq 0$ and we substitute $y$ with $y_1+y_2.$ We get
\[
\alpha_t(y_1+y_2)^tz(y_1+y_2)^{n-t-1} + \sum_{i < t}\alpha_i(y_1+y_2)^iz(y_1+y_2)^{n-i-1} \equiv 0 \pmod{\IsUT(UT_2)}.
\]
Let us consider the homogeneous component of degree $t$ in $y_1$ and $n- t -1$ in $y_2$. If we make the substitution $y_1=e_{11},$ $y_2=e_{22}$ and $z=e_{12}$, we obtain $\alpha_te_{12}=\alpha_t=0$, a contradiction. Hence the polynomials $a^{(i)}, 0\leq i\leq n-1$ are linearly independent$\pmod{\IsUT(UT_2)}$.
Therefore $m_{\lambda,\mu} \geq n.$

In the case of $\lambda=(p+q,p)$ and $\mu=(1)$ with $p\geq 1$, for any $0\leq i\leq q,$ we consider the following two tableaux:
\[
T_\lambda^{(i)}=\begin{tabular}{|c|c|c|c|c|c|c|c|c|c|}
\hline
$i+1$ & $i+2$ & $\cdots$ & $i+p$ & 1 & $\cdots$ & $i$ & $i+2p+2$ & $\cdots$ & $n$ \\
\hline
$i+p+2$ & $i+p+3$  & $\cdots$ & $i+2p+1$ & \multicolumn{5}{c}{} \\
\cline{1-4}
\end{tabular}
\]
\[
T_\mu^{(i)}=\begin{tabular}{|c|}
\hline
$i+p+1$ \\
\hline
\end{tabular}
\]
and the polynomials associated with them:
$$
    b_{p,q}^{(i)}(y_1,y_2,z)=y_1^i\underbrace{\overline{y}_1\cdots \tilde{y}_1}_p z \underbrace{\overline{y}_2 \cdots \tilde{y}_2}_py_1^{q-i}
    $$
Since the previous highest weight vectors are exactly the same that arise in the computation of the ordinary $\mathbb Z_2$-graded cocharacter of $UT_2,$ following word by word the lines of \cite[Theorem 3]{VA}, it can be proven that these polynomials are linearly independent modulo $\IsUT(UT_2),$ so $m_{\lambda,\mu} \geq q+1.$

By \eqref{spezzetaton-11}, \eqref{gcrn-r} and \eqref{Rizzo}, we have proved that
$$
2^{n-1} = c_{n-1,1}^{\mathbb{Z}_2, UT_2}(UT_2)\geq nd_{(n-1)}d_{(1)}+\sum_{\substack{1 \leq p \leq \lfloor \frac{n-1}{2} \rfloor \\ 0 \leq q \leq n-2p-1}} (q+1) d_{(p+q,p)}d_{(1)} = n  +  2^{n-1}-n+1-1 = 2^{n-1},
$$
thus $m_{(n-1),(1)} = n$ and $m_{(p+q,p), (1)}= q+1.$ The proof is now complete.
\end{proof}

Combining the previous lemmas, we obtain the following result.

\begin{theorem}\label{cocarattere}
    Let $UT_2$ be the $\mathbb{Z}_2$-graded $W$-algebra of upper $2 \times 2$ triangular matrices over a field $F$ of characteristic zero endowed with the canonical grading where $W$ acts as the full algebra $UT_2$ by left and right multiplication. Let $0\leq r \leq n$ and let
    \[
    \chi_{r,n-r}^{\mathbb{Z}_2,W}(UT_2)=\sum_{\substack{\lambda \vdash r \\ \mu \vdash n-r}}m_{\lambda,\mu}(\chi_\lambda \otimes \chi_\mu)
    \]
    be the $(r,n-r)$th generalized graded cocharacter of $UT_2$. Then either $r=n$ and 
    \[   m_{\lambda,\mu}=\begin{cases}
         n+3 & \lambda=(n), \ \mu=\emptyset\\
         q+1 & \lambda=(p+q,p), \ \mu =\emptyset\\
     0 & \text{otherwise}
\end{cases}, \]
 or $r=n-1$ and 
\[
         m_{\lambda,\mu}=\begin{cases}
   n & \lambda=(n-1),\  \mu=(1)\\
 q+1 & \lambda=(p+q,p), \ \mu=(1) \\
     0 & \text{otherwise}
\end{cases}.
\]
\end{theorem}

\begin{proof}
    Notice that $z_1z_2\equiv 0,$ therefore $m_{\lambda, \mu}=0$ as long as $\mu\vdash k$ with $k\geq 2.$ Now the proof follows directly from Lemmas \ref{lemmam_1} and \ref{lemmam_2}.
\end{proof}

Now with similar techniques, we can also compute the $(r,n-r)$th graded generalized cocharacter of $UT_2^D$, i.e., $UT_2$ with canonical grading when $\Phi(W) \cong D=Fe_{11}\oplus Fe_{22}$.

\begin{theorem}
\label{TheoremD}
Let $UT_2^D$ be the $\mathbb{Z}_2$-graded $W$-algebra of upper $2 \times 2$ triangular matrices over a field $F$ of characteristic zero endowed with the canonical grading where $W$ acts as the algebra $D$ by left and right multiplication. Let  $0\leq r \leq n$ and let
    \[
    \chi_{r,n-r}^{\mathbb{Z}_2,W}(UT_2^D)=\sum_{\substack{\lambda \vdash r \\ \mu \vdash n-r}}m_{\lambda,\mu}(\chi_\lambda \otimes \chi_\mu)
    \]
    be the $(r,n-r)$th generalized graded cocharacter of $UT_2^D$. Then either $r=n$ and
 \[
         m_{\lambda,\mu}=\begin{cases}
         2 & \lambda=(n), \ \mu=\emptyset\\
     0 & \text{otherwise}
\end{cases},
\]
or $r=n-1$ and 
    \[
         m_{\lambda,\mu}=\begin{cases}
   n & \lambda=(n-1), \ \mu=(1)\\
 q+1 & \lambda=(p+q,p), \ \mu=(1) \\
     0 & \text{otherwise}
\end{cases}.
\]
\end{theorem}

Analogously, we can also compute the $(r,n-r)$th graded generalized cocharacter of $\mathbb{Z}_2$-graded $W$-algebra $UT_2^C$, i.e., $UT_2$ with canonical grading when $\Phi(W) \cong C=F1_{UT_2}+ Fe_{12}$.

\begin{theorem}
Let $UT_2^C$ be the $\mathbb{Z}_2$-graded $W$-algebra of upper $2 \times 2$ triangular matrices over a field $F$ of characteristic zero endowed with the canonical grading where $W$ acts as the algebra $C$ by left and right multiplication. Let $0\leq r \leq n$ and let
    \[
    \chi_{r,n-r}^{\mathbb{Z}_2,W}(UT_2^C)=\sum_{\substack{\lambda \vdash r \\ \mu \vdash n-r}}m_{\lambda,\mu}(\chi_\lambda \otimes \chi_\mu)
    \]
    be the $(r,n-r)$th generalized graded cocharacter of $UT_2^C$. Then either $r=n$ and
 \[
         m_{\lambda,\mu}=\begin{cases}
         n+2 & \lambda=(n), \ \mu=\emptyset\\
     0 & \text{otherwise}
\end{cases},
\]
or $r=n-1$ and 
    \[
         m_{\lambda,\mu}=\begin{cases}
   n & \lambda=(n-1), \ \mu=(1)\\
 q+1 & \lambda=(p+q,p), \ \mu=(1) \\
     0 & \text{otherwise}
\end{cases}.
\]
\end{theorem}

Finally, for the sake of completeness, we also state the result for $UT_2^F$, i.e., $UT_2$ with the canonical grading and with $\Phi(W)\cong F$, whose proof can be found in \cite{VA}.

\begin{theorem}
Let $UT_2^F$ be the $\mathbb{Z}_2$-graded $W$-algebra of upper $2 \times 2$ triangular matrices over a field $F$ of characteristic zero endowed with the canonical grading where $W$ acts as the algebra $F$ by left and right multiplication. Let $0\leq r \leq n$ and let
    \[
    \chi_{r,n-r}^{\mathbb{Z}_2,W}(UT_2^F)=\sum_{\substack{\lambda \vdash r \\ \mu \vdash n-r}}m_{\lambda,\mu}(\chi_\lambda \otimes \chi_\mu)
    \]
    be the $(r,n-r)$th generalized graded cocharacter of $UT_2^F$. Then either $r=n$ and
 \[
         m_{\lambda,\mu}=\begin{cases}
         1 & \lambda=(n), \ \mu=\emptyset\\
     0 & \text{otherwise}
\end{cases},
\]
or $r=n-1$ and 
    \[
         m_{\lambda,\mu}=\begin{cases}
   n & \lambda=(n-1), \ \mu=(1)\\
 q+1 & \lambda=(p+q,p), \ \mu=(1) \\
     0 & \text{otherwise}
\end{cases}.
\]
\end{theorem}

\section{On almost polynomial growth}
In this section we deal with generalized varieties of superalgebras of almost polynomial growth (APG). To this end, recall that a variety of superalgebras $\mathcal{V}$ is said to be of almost polynomial growth of the codimensions, if $\gcn(\mathcal{V})$ grows exponentially, but for any proper subvariety $\mathcal{U}\subsetneq\mathcal V,$ $\gcn(\mathcal{U})$ grows polynomially.

Concerning this topic, given a class of algebras (associative, non-associative, with additional structure, etc), a central problem in the theory is twofold: on one hand, to construct algebras that are not PI-equivalent but still generate APG varieties; on the other hand, to classify, up to equivalence, all varieties exhibiting almost polynomial growth. For instance, a celebrated theorem by Kemer states that in characteristic zero up to equivalence, the only APG varieties of ordinary algebras are the ones generated by $UT_2$ and by the infinite dimensional Grassmann algebra $E$ (see \cite{Kemer1979}). Following this approach, in \cite{MR} it was proved that the generalized varieties of ordinary algebras generated by $UT_2^F$ and $UT_2^D$ are APG, while in \cite{MR2} the authors showed that up to equivalence, these are the only generalized APG varieties generated by finite dimensional algebras.

Motivated by the above results, in this section we aim to prove that in case of superalgebras, the generalized varieties generated by $UT_2^F$ and by $UT_2^D,$ both endowed with the canonical grading, are APG.

\begin{remark}
\label{remarkutd}
Recall that by Theorem \ref{TheoremD}, if $\lambda=(n)$ and $\mu=\emptyset$ then the linearly independent graded generalized highest weight vectors corresponding to $\lambda$ and $\mu$ are
    \[
    y^n \quad \text{ and } \quad
    w_1y^n,
    \]
    where $w_1$ acted on $UT_2$ as the left and right multiplication by $e_{22}.$ Moreover, if $\lambda=(n-1)$ and $\mu=(1)$ then the linearly independent graded generalized highest weight vectors are
    \[
    y^izy^{n-i-1}, \quad 0\leq i \leq n-1.
    \]
   Finally, for all $p\geq 1$, $q\geq 0$, $2p+q=n-1$, if $\lambda=(p+q,p)$ and $\mu=(1),$ the linearly independent graded generalized highest weight vectors corresponding to $\lambda$ and $\mu$ are
    \[
    y_1^i\underbrace{\overline{y}_1\cdots \tilde{y}_1}_p z \underbrace{\overline{y}_2 \cdots \tilde{y}_2}_py_1^{q-i}, \quad 0\leq i\leq q.
    \]
\end{remark}
\begin{theorem}
\label{polynomialgrotwD}
    $\VsupW (UT_2^D)$ has almost polynomial growth of the codimensions.
\end{theorem}

\begin{proof}

     Let $A$ be a $\mathbb{Z}_2$-graded $W$-algebra such that $\IsupW(UT_2^D) \subsetneq \IsupW(A).$ 
    For $n\geq 1$ and $0\leq r \leq n$, denote by $m_{\lambda,\mu}(A)$  and $m_{\lambda,\mu}(UT_2^D)$ the multiplicity of the generalized graded $(r,n-r)$th cocharacter of $A$ and $UT_2^D,$ respectively, corresponding to the partitions $\lambda\vdash r$ and $\mu\vdash n-r.$ We claim that there exists $N\geq 0$ such that $m_{\lambda,\mu}(A)\leq N$ for all $\lambda\vdash r$ and $\mu\vdash n-r.$ Moreover, if $\lambda=(p+q,p)$ with $p\geq 1$ and $q\geq 0,$ then there exists $M\geq 0$ such that $m_{\lambda,\mu}(A)=0$ when $p\geq M.$

    Since $\IsupW(UT_2^D) \subsetneq \IsupW(A),$ there exist $ n_0 \geq 1,$ $0\leq  r_0 \leq n_0,$ and partitions $\lambda'\vdash r_0$ and $\mu'\vdash n_0-r_0$ such that $m_{\lambda',\mu'}(A)< m_{\lambda',\mu'}(UT_2^D).$ By Remark \ref{remarkutd}, 
    It follows that either
\begin{equation}
\label{eq: HWV y}
\alpha_1y^{n_0}+\alpha_2w_1y^{n_0}\equiv 0 \pmod{\IsupW(A)} 
\end{equation}
with $\alpha_1,\alpha_2$ not both zero, or
\begin{equation}
\label{eq: HWV lambda=n-1}
    \sum_{i=0}^{n_0-1} \beta_iy^izy^{n_0-i-1} \equiv 0 \pmod{\IsupW(A)}
\end{equation}
with $\beta_i$ not all zero, or
\begin{equation}
\label{eq: HWV (p+q,p)}
 \sum_{i=0}^{q_0} \gamma_iy_1^i\underbrace{\overline{y}_1\cdots \tilde{y}_1}_{p_0} z \underbrace{\overline{y}_2 \cdots \tilde{y}_2}_{p_0}y_1^{q_0-i} \equiv 0 \pmod{\IsupW(A)}
\end{equation}
with $\gamma_i$ not all zero, and  where $p_0\geq 1,$ $q_0\geq 0$, $2p_0+q_0=n_0-1$. 

 If \eqref{eq: HWV lambda=n-1} or \eqref{eq: HWV (p+q,p)} hold, then since such highest weight vectors are exactly the same of the superalgebra $UT_2$ in the ordinary case, the claim follows directly from the proof of \cite[Theorem 4]{VA}.

 Suppose now that \eqref{eq: HWV y} holds.
 Since $zw_1\equiv z  \pmod{\IsupW(UT_2^D)},$ for $\alpha_1\neq - \alpha_2$ from \eqref{eq: HWV y} we get that
 \begin{equation} \label{eq: alpha1 alpha2}
     zy^{n_0} \equiv 0 \pmod{\IsupW(A)}.
 \end{equation}
Moreover, since $w_1y\equiv yw_1  \pmod{\IsupW(UT_2^D)}$ and $w_1 z\equiv 0 \pmod{\IsupW(UT_2^D)},$ for $\alpha_1\neq 0$ from \eqref{eq: HWV y} we have that
 \begin{equation} \label{eq: alpha1 0}
     y^{n_0}z \equiv 0 \pmod{\IsupW(A)}.
 \end{equation}

 Notice that by Theorem \ref{TheoremD}, to reach our goal it is enough to consider only the cases when $\lambda=(n-1)$ and $\mu=(1)$ or $\lambda=(p+q,p)$ and $\mu=(1).$

 First suppose that $\lambda=(n-1)$ and $\mu=(1).$ If $n\leq n_0,$ then again by Theorem \ref{TheoremD}, $m_{\lambda, \mu}(A)\leq n
  \leq n_0$ and we have nothing to prove. 

Now, assume that $n> n_0$. If $\alpha_1\neq - \alpha_2$, then due to the identity \eqref{eq: alpha1 alpha2}, it follows that $y^izy^{n-i-1}\equiv 0  \pmod{\IsupW(A)}$ as long as $n-i-1\geq n_0.$ Therefore for all $0\leq i\leq n-n_0-1$ the polynomials $y^izy^{n-i-1}$ are $\mathbb{Z}_2$-graded $W$-identities of $A$ and, since the remaining highest weight vectors correspond to the index $i$ running in the range from $n-n_0$ to $n-1,$ the maximal number of linearly independent highest weight vectors is at most $n_0,$ i.e., $m_{\lambda, \mu}(A)\leq n_0$ for all $\lambda=(n-1),$ $\mu=(1)$ and $n\geq 2.$

Now, if $\alpha_1=-\alpha_2,$ then $\alpha_1\neq 0$ because $\alpha_1, \alpha_2$ are not both zero. Thus by the identity \eqref{eq: alpha1 0}, we get that $y^izy^{n-i-1}\equiv 0  \pmod{\IsupW(A)}$ for all $i\geq n_0$. Since the remaining highest weight vectors correspond to the index $i$ running in the range from $0$ to $n_0-1,$ again as before, the maximal number of linearly independent highest weight vectors is at most $n_0.$ Hence $m_{\lambda, \mu}(A)\leq n_0$ for all $\lambda=(n-1),$ $\mu=(1)$ and $n\geq 2.$

We now focus our attention on the polynomials associated to partitions of the type $\lambda=(p+q,p)$ and $\mu=(1).$ First notice that, since $[y_1,y_2]\equiv 0 \pmod{\IsupW(UT_2^D)},$ if $p\geq 2n_0,$ then from either \eqref{eq: alpha1 alpha2} or \eqref{eq: alpha1 0} it follows that
$$
\underbrace{\overline{y}_1\cdots \tilde{y}_1}_p z \underbrace{\overline{y}_2 \cdots \tilde{y}_2}_p \equiv 0 \pmod{\gids(A)}
$$
because, either before or after the variable $z$, there are at least $n_0$ copies of either $y_1$ or $y_2$. Therefore $m_{\lambda, \mu}(A)=0$ when $p\geq 2n_0.$
Thus, we can assume that $p<2n_0.$ 

Now, if $q<n_0,$ then by Theorem \ref{TheoremD}, $m_{\lambda, \mu}(A)\leq q+1\leq  n_0$ and we have nothing to prove. Then assume that $q\geq n_0.$ If $\alpha_1\neq - \alpha_2$, then from \eqref{eq: alpha1 alpha2} and since $[y_1,y_2]\equiv 0 \pmod{\IsupW(UT_2^D)},$ it follows that 
$$
y_1^i\underbrace{\overline{y}_1\cdots \tilde{y}_1}_p z \underbrace{\overline{y}_2 \cdots \tilde{y}_2}_py_1^{q-i} \equiv 0 \pmod{\gids(A)}
$$ 
for all $0\leq i\leq q-n_0.$ Since the remaining highest weight vectors correspond to the index $i$ running in the range from $q-n_0+1$ to $q,$ the maximal number of linearly independent highest weight vectors is at most $n_0,$ i.e., $m_{\lambda, \mu}(A)\leq n_0$ for all $\lambda=(p+q,p),$ $\mu=(1)$ and $p\geq 1$, $q\geq 0$.

Now, if $\alpha_1=-\alpha_2,$ then $\alpha_1\neq 0$ because $\alpha_1, \alpha_2$ are not both zero. Thus by using \eqref{eq: alpha1 0} and $[y_1,y_2]\equiv 0 \pmod{\IsupW(UT_2^D)},$ we get as above that $m_{\lambda, \mu}(A)\leq n_0 $ for all $\lambda=(p+q,p),$ $\mu=(1)$ and $p\geq 1$, $q\geq 0$.

In conclusion, in all cases we have proved that $m_{\lambda,\mu}(A) \leq N= n_0$ for all partitions $\lambda,$ $\mu,$ and $m_{\lambda,\mu}(A)=0$ for all $\lambda$ with more than $M=2n_0$ boxes below the first row, as claimed. Therefore
		$$
		\chi_{r,n-r}^{\mathbb{Z}_2, W}(A) = \sum_{\substack{\lambda\vdash r,\, \mu\vdash n-r\\ r-\lambda_1 \leq M-1\\ n-r \leq 1}} m_{\lambda,\mu}(A) \,(\chi_\lambda \otimes \chi_\mu)
		$$
		where $\lambda_1$ stands for the number of boxes of the first row of $\lambda.$
		
		Since $r-\lambda_1\leq M-1,$ then $\lambda_1\geq r-(M-1) $ and by the hook formula we get 
		$$
		d_\lambda = \chi_\lambda(1) \leq \frac{n!}{(r-(M-1))!}\leq n^{M-1} .
		$$
        Thus
        \begin{align*}
            c^{\mathbb{Z}_2, W}_n(A)&=\sum_{r=0}^n \binom{n}{r} \chi_{r,n-r}^{\mathbb{Z}_2, W}(A)(1)= \sum_{r=0}^n \binom{n}{r} 
        \sum_{\substack{\lambda\vdash r,\, \mu\vdash n-r\\ r-\lambda_1 \leq M-1\\ n-r \leq 1}} m_{\lambda,\mu}d_\lambda d_\mu \\
            &=n \sum_{\substack{\lambda\vdash n-1,\\ n-\lambda_1 \leq M\\ \mu= (1)}} m_{\lambda,\mu}d_\lambda d_\mu  + 2\leq n \sum_{\substack{\lambda\vdash n-1,\\ n-\lambda_1 \leq M\\ \mu= (1)}}  N n^{M-1} +2\leq 2(M-1)^2 N n^{M}+2.
        \end{align*}
The latter inequality holds since the number of partitions $\lambda\vdash n-1$ such that $n - \lambda_1 \leq M$ is bounded by $(M-1)^2$. Therefore $c^{\mathbb{Z}_2, W}_n(A)$ is polynomially bounded and $\VsupW(UT_2^D)$ has almost polynomial growth of the codimensions, as required.
\end{proof}

With respect to the superalgebra $UT_2^F,$ recall that in this case $\Phi(W)\cong F,$ therefore we are dealing with the ordinary graded identities of $UT_2$ and by \cite[Theorem 4]{VA} we have the following.
\begin{theorem}
\label{polynomialgrowtF}
    The variety of $\mathbb{Z}_2$-graded $W$-algebras generated by $UT_2^F$ has almost polynomial growth.
\end{theorem}

  As consequence of Theorems \ref{TheoremUT_2^D} and \ref{TheoremUT2F}, we have that $\gids(UT_2^F) \not\subseteq \gids(UT_2^D)$ and $\gids(UT_2^D)\not\subseteq \gids(UT_2^F)$. Thus, by Theorems \ref{polynomialgrotwD} and \ref{polynomialgrowtF} we have the following.
\begin{corollary}
    The superalgebras $UT_2^F$ and $UT_2^D$  generate two distinct varieties of $\mathbb{Z}_2$-graded $W$-algebras of almost polynomial growth.
\end{corollary}

Finally, notice that by Theorems \ref{Tideale} and \ref{TheoremUT_2^D},  $UT_2^D \in \VsupW (UT_2)$ and $\VsupW (UT_2^D)$ grows exponentially. Moreover, by Theorems \ref{AlgebraCidentità} and \ref{TheoremUT2F}, $UT_2^F \in \VsupW (UT_2^C)$ and $\VsupW (UT_2^F)$ grows exponentially. Then it immediately follows that 
\begin{corollary}
    $\VsupW  (UT_2)$ and $\VsupW (UT_2^C)$ do not have almost polynomial growth of the codimensions.
\end{corollary}

\end{document}